\documentclass[11pt,a4paper]{amsart}
\usepackage{a4wide}
\usepackage{amsfonts,amsthm,amsmath,amssymb}
\usepackage{amscd,color,caption}
\usepackage{psfrag,subfigure,float,subfloat}
\usepackage{import,algorithmic,algorithm2e}
\usepackage[dvipdfm]{graphicx}

\newtheorem*{thmA}{Theorem A}
\newtheorem*{thmB}{Theorem B}
\newtheorem*{thmC}{Theorem C}

\theoremstyle{plain}

\newtheorem{theorem}{Theorem}[section]
\newtheorem{lemma}[theorem]{Lemma}

\newtheorem{remark}{Remark}

\newcommand{\cal}{\mathcal}

\newcommand{\w}{\color{white}}
\newsavebox{\savepar}

\begin{document}

\title[Dynamics of the  Secant Method]{
Global dynamics of the real secant method}

\date{\today}

\author{Antonio  Garijo}
\address{Departament d'Enginyeria Inform\`atica i Matem\`atiques,
Universitat Rovira i Virgili, 43007 Tarragona, Catalonia.}
\email{antonio.garijo@urv.cat}
\author{Xavier Jarque}
\address{Departament de Matem\`atiques i Inform\`atica at Universitat de Barelona and Barcelona Graduate School of Mathematics, 08007 Barcelona, Catalonia.}
\email{xavier.jarque@ub.edu}

\thanks{This work has been partially supported by MINECO-AEI grants
MTM-2017-86795-C3-2-P and MTM-2017-86795-C3-3-P, and the AGAUR grant 2017 SGR 1374}\

\begin{abstract}
We investigate the root finding algorithm given by the secant method applied to a real polynomial $p$ as a discrete dynamical system defined on $\mathbb R^2$. We study the shape and distribution of the basins of attraction associated to the roots of $p$, and we also show the existence of other stable dynamics that might affect the efficiency of the algorithm. Finally we extend the secant map to the punctured torus  $\mathbb T^2_{\infty}$ which allow us to better understand the dynamics of the secant method near $\infty$ and facilitate the use of the secant map as a method to find all roots of a polynomial. 
\vspace{0.5cm}


\end{abstract}

\maketitle

\section{Introduction}\label{sec:Introduction}

Root finding algorithms
\begin{equation}\label{eq:rootfinding}
x_{n+1}=R\left(x_{n-\ell},\ldots, x_n\right), \ \ell \geq 0, n \geq \ell, \ x\in X
\end{equation} 
are iterative systems so that for most initial {\it seeds} $\left(x_0,\ldots, x_{\ell}\right)$ the sequence $\{x_n\}_{n\geq 0}$ converges to a  solution of a given nonlinear equation, namely $F(x)=0,\ x\in X$. Since many real problems can be modelled in terms of nonlinear equations which do not admit explicit solutions, the applicability of those algorithms is wide over all areas like engineering, economics, sociology or biology. Through the whole paper we will assume that $F$ is a polynomial $p$ of degree at least 2 and hence our goal is to solve the equation $p(x)=0, \ x \in X$ with $X=\{\mathbb R,\mathbb C\}$.

There are many natural questions about the efficiency of those algorithms. If the equation has more than one solution, as it happens in most cases, how to find different seeds converging to each of them? Is it possible to have regions with positive measure where seeds do not  converge to any solution of $p(x)=0$? What can be said about the speed (number of steps) to have a reasonable approximation of the solution (order of convergence)? To answer these questions, or more ambitious, to have a deeper understanding of those algorithms, we treat them as discrete dynamical systems (see, for instance, \cite{HowToNewton}). 

Roughly speaking a discrete dynamical system over $Y$ is a map $f:Y\to Y$ and the orbits induced by this map starting at  $y_0\in Y$, that is, $\{y_n:=f^n\left(y_0\right)\}_{n \geq 1}$ where 
$$
f^n = \underbrace{f\circ \cdots \circ f}_{n \textrm{ times}}.
$$
A main goal when studying dynamical systems is to describe the asymptotic behaviour of those orbits when $y_0$ runs over all $Y$. In particular the study of {\it fixed points}, i.e. $y_0$ in $Y$ such that $f\left(y_0\right)=y_0$. Those points are classified according to the behaviour of the nearby seeds and play a key role on the global dynamics. In particular, a fixed point $y_0$ in $Y$ is called {\it attracting} if there exists $\varepsilon>0$ such that $f^n(y)\to y_0$, as $n \to \infty$, for all $y$  in 
$D_\varepsilon(y_0):=\{y\in Y \ | \ d\left(y,y_0\right) < \varepsilon \}$. Accordingly, given an attracting fixed point   $y_0\in Y$, its {\it basin of attraction} is denoted by $A\left(y_0\right)$ and given by
$$
A\left(y_0\right)=\{y\in Y \ | \ f^{n}(y)\to y_0, \ {\rm as} \  n\to \infty  \},
$$
which it is open and nonempty  by definition.  The connected component of $A\left(y_0\right)$  which contains $y_0$ is called the {\it immediate basin of attraction} of $y_0$  and it is denoted by $A^{\star}\left(y_0\right)$. We also consider {\it periodic points of (minimal) period  $q$}, or {\it $q$-periodic points}, i.e. $y_0$ in $Y$ such that there exist $q\geq 2$ satisfying
$f^q\left(y_0\right)=y_0$ and $f^\ell\left(y_0\right)\ne y_0$ for all $\ell<q$. Similarly, we can define attracting $q$-periodic points and their attracting basins.  

Therefore a root finding algorithm (\ref{eq:rootfinding}) for the equation $p(x)=0,\ X=\{\mathbb R,\mathbb C\}$ can be seen as a discrete dynamical system generated by the map
$$
f_p:Y\to Y,\quad Y=X\times \overset{\ell)}{\cdots} \times X
$$
for which the orbits $\{y_n:=f_p^n\left(y_0\right)\}_{n\in \mathbb N}$ converge, for {\it most} initial conditions, to fixed points of $f_p$ which are in correspondence to the roots of $p$.  The most well-known and universal root finding algorithm applied to $p(z)=0,\ z\in \mathbb C$ is the so called {\it Newton's method}
\begin{equation*}\label{eq:Newton}
N_p: \mathbb C \mapsto \mathbb C, \qquad N_p(z)=z - \frac{p(z)}{p'(z)}.
\end{equation*}
Easy computations show that if $\zeta\in \mathbb C$ is a simple root of $p$ then $N_p(\zeta)=\zeta$ and $N_p^{\prime}(\zeta)=0$, so $\zeta$ is an attracting fixed point. It turns out that for {\it most} initial conditions $z_0\in \mathbb C$ the sequence $\{z_{n}:=N_p^n\left(z_0\right)\}_{n\geq 1}$ converges to some root of $p$. Certainly the dynamical system is not well defined at the critical points of $p$ since the denominator of $N_p$ vanishes but we go over this problem by extending the phase space from $\mathbb C$ to $\hat{\mathbb C}=\mathbb C \cup \{ \infty\}$, where   $\hat{\mathbb C}$ denotes the Riemann sphere.  One can show by the use of the charts defined on $\hat{\mathbb{C}}$ that the Newton's map is well defined at the whole Riemann sphere, and in particular $\infty$ is always a repelling fixed point of $N_p$.

The literature on Newton's method as a root finding algorithm as well as dynamical system is extremely large and in fact  it was the starting point of {\it holomorphic dynamics} (see \cite{cayley1,cayley2,cayley3}). For instance we refer to M. Shishikura \cite{ConnectivityJulia} who proved a remarkable theorem implying  the simple connectivity of the immediate basins of attraction (see also \cite{przytycki}), and we refer to J. Hubbard, D. Schleicher and S. Sutherland \cite{HowToNewton}, who provided a universal set (only depending on the degree of the polynomial) of initial conditions to find out all roots of a polynomial of a given degree. As a counterpart it is known that for certain polynomials of degree larger than two there are open sets of initial conditions for which $N_p$ do not converge to any root of $p$ (see \cite{FamiliesRational}) and so Newton's method for those polynomials might fail as a root finding algorithm. Finally we refer to  \cite{ConnectivityMero} and \cite{BarFagJarKar} for Newton's method applied to transcendental maps.

Alternative to Newton's method, another well known root finding algorithm is the  secant method, although few references can be found in the literature. The main goal of this paper is twofold. On the one hand we explore the secant method as a dynamical system (iterates of the secant map) and on the other hand we provide arguments for a better implementation of the secant method as a root finding algorithm. The {\it secant method} is given by the iterates of the {\it secant map}   
\begin{equation*}
S:=S_p : \mathbb C^2 \mapsto \mathbb C^2 , \qquad
S :
\left( 
\begin{array}{l}
z \\
w
\end{array}
\right) \mapsto  
 \left( 
\begin{array}{l}
w \\
w - p(w) \frac{w-z}{p(w)-p(z)}
\end{array}
\right).
\end{equation*}
It is worthily notice that the Newton's map $N_p$ defines a dynamical system on $\mathbb C$ (that is, we have $\ell=0$ in (\ref{eq:rootfinding})) while the secant map  $S_p$ defines a dynamical system on $\mathbb C^2$ (that is, we have $\ell=1$ in (\ref{eq:rootfinding})).  Of course this fact makes the general study more difficult.
A first step towards the understanding of the complexity of this dynamical system is to restrict the attention to the {\it real version} of the secant method. 

More precisely, let 
\begin{equation}\label{eq:poly_p}
p(x)= \sum_{j=0}^k a_j x^j, \ a_k=1,\ a_j\in \mathbb R,
\end{equation}
with $k>1$, be a monic polynomial having exactly $n$ real roots denoted by $\alpha_1<\alpha_2<\ldots <\alpha_n$ with $1\leq n \leq k$, all simple. The (real) secant map is given by
\begin{equation} \label{eq:secant_real}
S:=S_p: \mathbb R^2 \mapsto \mathbb R^2 , \qquad
S: \left( 
\begin{array}{l}
x \\
y
\end{array}
\right) \mapsto  
 \left( 
\begin{array}{l}
y \\
y - p(y) \frac{y-x}{p(y)-p(x)}
\end{array}
\right),
\end{equation}
and the orbit of the seed $\left(x_0,y_0\right)\in \mathbb R^2$  is given by the iterates of the map; that is, the sequence $\{(x_n,y_n)=S^n\left(x_0,y_0\right)\in \mathbb R^2\}_{n\in \mathbb N}$. 

We notice that the real version of Newton's method and secant method is based on a similar idea. In Newton's method, for a given seed $x_0\in \mathbb R$, the point $x_1=N_p\left(x_0\right)$ is the intersection between the $x-$axis and the tangent line through the point $\left(x_0,p(x_0)\right)$ while in the secant case for a given seed $(x_0,y_0)\in \mathbb R^2$, the point $(x_1,y_1)=S\left(x_0,y_0\right)$ is given by $x_1=y_0$ and $y_1$ is the intersection between the $x-$axis and the secant line through the points $\left(x_0,p(x_0)\right)$ and $\left(y_0,p(y_0)\right)$. 

Certainly, the secant method has disadvantages with respect to Newton's method, like for instance that locally, near simple roots, Newton's method has quadratic convergence while the secant method has local convergence less than 2. However,
if we are only interested on the real roots of $p$, which is very plausible in multiple applications, the secant method studied in this paper might be a more powerful tool than  Newton's method. If for instance the polynomial has most of the roots complex and we use the algorithm developed in \cite{HowToNewton} we are making a tremendous numerical effort to compute roots which are not in our interest. 

For a review on numerical analysis and root finding algorithms (local order of convergence, computing implementation, etc) we refer to \cite{Traub_Book}. While this paper has been written we learned about \cite{BedFri}  where the authors have studied independently the secant map as a dynamical system on $\mathbb C^2$ (see Theorem C for a further discussion). Finally in \cite{secant_2} several root finding algorithms, including the secant method, are applied to degree two polynomials. 

\subsection*{Rational plane maps}
There are several papers in the literature studying discrete dynamical systems on $\mathbb C^2$ or $\mathbb {CP}^2$, induced by injective maps. For example there are many papers on polynomial automorphisms of $\mathbb C^2$ (see for instance \cite{BerfordSmillie1,BerfordSmillie2,BerfordSmillie3, DujardinHenon, DujardinLyubich}), on the complex version of the (polynomial) H\'enon map (see for instance\cite{HubbardHenon1,HubbardHenon2,HubbardHenon3,BerfordSmillie, FornaesSibonyHenon,BeniniHenon}), or birational maps (see \cite{CimaManosasBirational,BedfordDillerBirational,BedfordBirational,ZafarCima}.  In contrast the secant map defined on $\mathbb R^2$  or  $\mathbb C^2$ is not an injective map, and so most of the tools used on the above papers fail in this case. Accordingly, the natural framework for  studying $S$ as a plane dynamical system is the iteration of rational-like maps on $\mathbb R^2$ (see \cite{PlaneDenominator1,PlaneDenominator2,PlaneDenominator3}, and references therein, for a more complete discussion). We introduce here the notation we need to state our main results. Consider the family of maps (which include the secant map $S$)

\begin{equation}
T: \left( 
\begin{array}{l}
x \\
y
\end{array}
\right) \mapsto  
 \left( 
\begin{array}{l}
F(x,y) \\
N(x,y)/D(x,y)
\end{array}
\right),
\label{eq:PlaneDenominator}
\end{equation}
where $F$, $N$ and $D$ are differentiable functions. Set  
\[
\delta_T = \{ (x,y) \in \mathbb R^2 \, | \, D(x,y)=0\} \quad {\rm and} \quad E_T =\mathbb R^2 \setminus \bigcup_{n \geq 0} T^{-n}(\delta_T).
\]
Easily $E_T$ defines a natural subset of $\mathbb R^2$ where all iterates of  $T$ are well-defined, and so $T:E_T\to E_T$ defines a smooth dynamical system. In contrast, roughly speaking, $T$ {\it sends} points of $\delta_T$ to infinity since the denominator $D$ is zero, except at those points of $\delta_T$ where also the numerator $N$ is zero and hence the value of $T$ is uncertain.

Denote by $T_2$ second component of $T$. We say that a point  $Q$ in $ \delta_T$ is a  {\bf focal point} if $T_2$ evaluated at $Q$ takes the form 0/0 (i.e. $N(Q)=D(Q)=0$), and there exists a smooth simple arc $\gamma:=\gamma(t),\ t\in (-\varepsilon,\varepsilon)$, with $\gamma(0)=Q$, such that $\lim_{t \to 0} T_2(\gamma)$ exists and it is finite. Moreover, the line $L_Q=\{(x,y)\in \mathbb R^2 \ | \ x=F(Q)\}$ is called the {\bf prefocal line}. If we assume that $\gamma$ passes through $Q$,  not tangent to $\delta_T$, with slope $m$ (that is $\gamma^\prime(0)=m$), then $T\left(\gamma\right)$ will be a curve passing through some point  $(F(Q),y(m)) \in L_Q$ (at $t=0$). Precisely
\begin{equation}\label{eq:limit_m}
y(m) = \lim_{t \to 0} \frac{N(\gamma(t))}{D(\gamma(t))}.
\end{equation} 
In particular $T$, as a map from $\mathbb R^2$ to $\mathbb R^2$, is not continuous at the focal points. See Figure \ref{fig:focals}. However next result shows that the relation between the slope $m=\gamma^\prime(0)$ and the point $y(m)\in L_Q$ is a continuous and one-to-one map.

\begin{figure}[ht]
    \centering
     \includegraphics[width=0.75\textwidth]{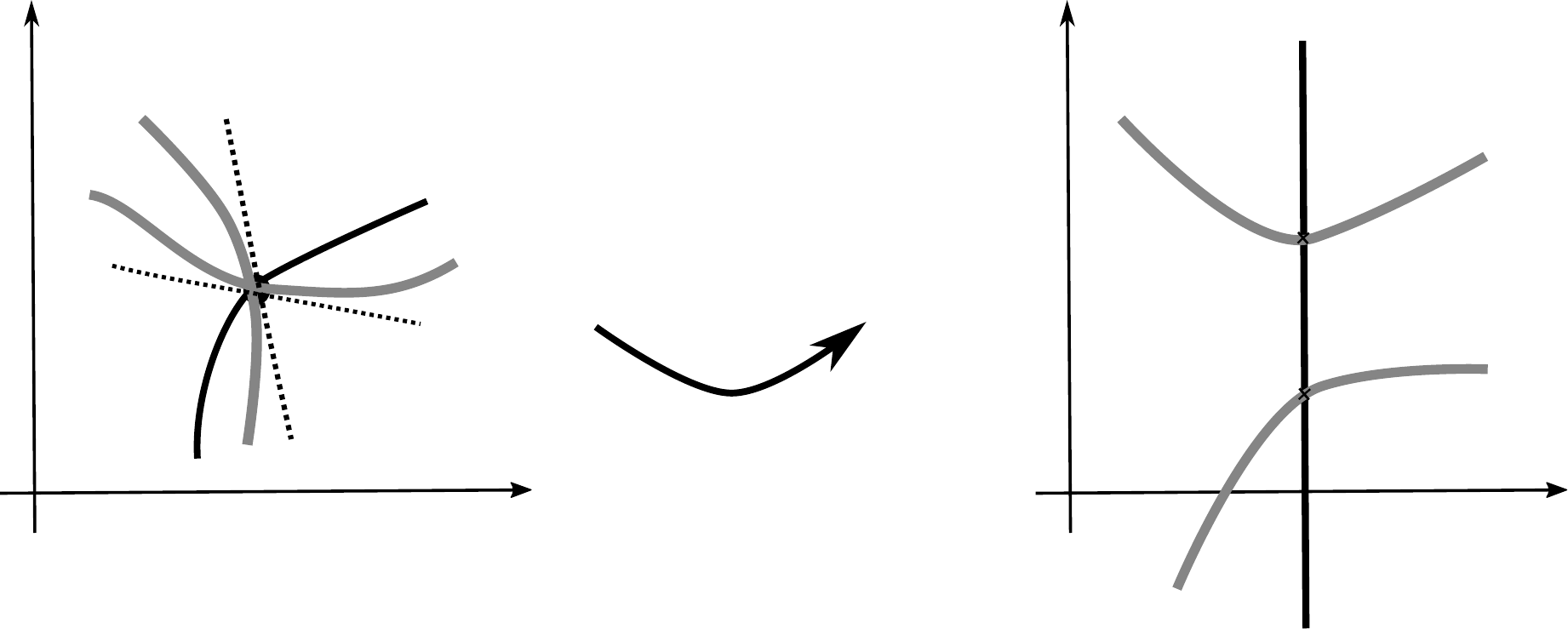}
    \put(-180,38) {\small $T$ } 
     \put(-305,112){\small $\gamma_1$}
     \put(-325,92){\small $\gamma_2$}
     \put(-255,58){\small $m_1$}
     \put(-275,35){\small $m_2$}
     \put(-273,60){\small $Q$}
     \put(-53,77){\small $y(m_1)$}
     \put(-53,42){\small $y(m_2)$}
     \put(-250,93){\small $ \delta_T\, [D(x,y)=0]$}
     \put(-15,100){\small $T(\gamma_1$)}
     \put(-12,52){\small $T(\gamma_2$)}
      \put(-50,6){\small $L_Q \, [  x = F(Q) ] $}
           \caption{\small{Discontinuity of $T$ at a focal point $Q$.}}
    \label{fig:focals}
    \end{figure}
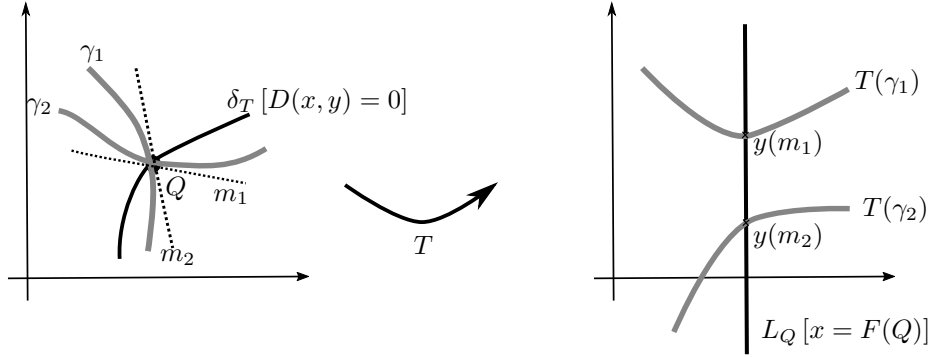

\begin{theorem}[{\cite{PlaneDenominator1}}]
\label{th:Gardini}
Let $T$ be the rational map described in (\ref{eq:PlaneDenominator}). Let $Q$ be one of its focal points and assume $N_x(Q)D_y(Q)-N_y(Q)D_x(Q)\neq 0$. Then there is a one-to-one correspondence between the slopes $m$ of an arc $\gamma$ through $Q$ not tangent to $\delta_T$, and the points $(F(Q),y)\in L_Q$. The correspondence writes as 
\begin{equation}
\begin{split}
& m \mapsto (F(Q),y(m)) \quad {\rm with} \quad y(m)= \frac{N_x(Q) + m N_y(Q)}{D_x(Q)+mD_y(Q)} \\
&(F(Q),y) \mapsto m(y) \quad {\rm with} \quad m(y)=\frac{D_x(Q) y- N_x(Q)}{N_y(Q)-D_y(Q)y}
\end{split}
\label{eq:y(m)}
\end{equation}
\end{theorem}

We denote by $\cal Q_T$ the set of all focal points of $T$ given by  (\ref{eq:PlaneDenominator}). Notice that $\cal Q_T$ as well as all prefocal lines belong to $\delta_T$ but they play a key role on the understanding of the global dynamics of the dynamical system generated by $T$ on $E_T$.  We are ready to state our main results.

\subsection*{Statement of the main results} 
Theorem A is about the shape and distribution of the basins of attraction of the fixed points of $S$, in particular we show that any focal point belongs to the boundary of the  basin of attraction of all the roots of the polynomial $p$ (see (\ref{eq:poly_p})). 
Theorem A shows that the focal points are surrounded by initial seeds corresponding to all basins of attractions (statement (d)). So, nearby focal points would be a natural place to find out good seeds converging to {\it all} roots of $p$, an important issue from the practical use of the secant method as a root finding algorithm. We do not ignore (see statement (c)) that focal points are related to the roots of $p$ (which we do not know!).  However notice that the smallest ($\alpha_1$) and largest ($\alpha_n$) real roots of $p$ has unbounded immediate basin of attraction (see statement (b)) which make them easy to compute numerically. Moreover the corresponding focal point $Q_{1,n}=(\alpha_1,\alpha_n)$ (statement (c)) is of special interest because the {\it tunnels size} of the attracting basins next to it are significantly width. Further work on this direction is in process.

\begin{thmA}\label{theo:basins}
The secant map $S$ induces a smooth dynamical system on $E_S$. Moreover the following statements hold.
\begin{itemize}
\item[(a)] The only fixed points of $S$ are the points 
$\left(\alpha_\ell,\alpha_\ell\right),\ \ell=1,\ldots,n$, and they are all attracting.
\item[(b)] Each basin of attraction $A\left(\alpha_\ell\right),\ \ell=1,\ldots n$, is unbounded.  If $n=1$ or $n=2$, then $A^\star\left(\alpha_1\right)$ (and $A^\star\left(\alpha_2\right)$, if $n=2$) are unbounded. If $n\geq 3$ then $A^\star\left(\alpha_1\right)$ and $A^\star\left(\alpha_n\right)$ are unbounded while $A^\star\left(\alpha_\ell\right)$ with $2 \leq \ell \leq n-1$ are bounded.
\item[(c)] Let $Q_{i,j}:=(\alpha_i,\alpha_j)$ for $i,j=1\ldots n, \ i\ne j $. The set of focal points of the secant map is given by 
$$
\cal Q:=\cal Q_S = \bigcup_{\substack{i,j=1,\ldots , n  \\ i\ne j}}   Q_{i,j}  . 
$$
\item[(d)]Each focal point belongs to the common boundary of all basins of attraction, that is,
$$
\displaystyle Q_{i,j}\in \bigcap_{\ell  =1, \ldots, n}\partial A\left(\alpha_\ell\right) .
$$ 
\end{itemize} 
\end{thmA}

The second and third result of this paper (Theorem B and Theorem C) deal with the existence of (unwanted) stable dynamics; that is the existence of open regions on the dynamical plane where seeds do not converge to the fixed points of $S$ associated to the roots of the  polynomial $p$. In other words both results bound the efficiency of the secant map as a root finding algorithm (see also \cite{FamiliesRational}), again a dynamical result with a relevant impact on numerical computations.  

\begin{thmB}\label{theorem:noperiod2_3} 
Let $S:E_S\to E_S$ be the dynamical system induced by the secant map defined on  (\ref{eq:secant_real}). The following statements hold.
\begin{itemize}
\item[(a)] $S$ has no periodic orbits of minimal period either two or three. 
\item[(b)] There exists a polynomial $p^{\star}$ such that $S_{p^{\star}}$ exhibits an attracting periodic orbit of minimal period four. In particular the dynamical plane has open regions of initial conditions for which $S_{p^{\star}}$ does not converge to any root of the polynomial $p^{\star}$.
\end{itemize} 
\end{thmB}

In contrast to Theorem B where it is shown that $S$ (defined on $E_S$) has no periodic orbits of period two and three, Theorem C shows that {\it stable} period three cycles exist, for all polynomials $p$, if we extend the map to {\it infinity}. As we said seeds converging to this three cycle should be discard when using $S$ as a root finding algorithm. Here $\mathbb T^2_{\infty}$ is a torus minus one point (see Section \ref{sec:torus} for details).      
 
\begin{thmC}\label{theorem:bakers} 
The map $\hat{S}:\mathbb T^2_{\infty} \setminus \cal Q \mapsto \mathbb T^2_{\infty}\setminus \cal Q$ defines a smooth extension of $S$. Moreover if $x_0$ verifies $p'(x_0)=0$ then the point $(x_0,x_0)$ is periodic of minimal period three, namely $(x_0,x_0) \mapsto (x_0,\infty) \mapsto (\infty,x_0) \mapsto (x_0,x_0)$.  The eigenvalues of $D\hat{S}^3(x_0,x_0)$ are 0 and 1 if the degree of the polynomial $p$ is greater or equal than 3.
\end{thmC}

While this paper has been written we learned that Theorem C has been proved independently in \cite{BedFri} where the authors deal with the secant map on $\mathbb C^2$. In fact the authors determine an open region of $\mathbb C^2$ belonging to the attracting basin of the three cycle (see our Figure \ref{fig:3cycle} for a numerical evidence).

We organize the paper as follows. In Section \ref{sec:FinitePlane} we prove Theorems A and B. In Section \ref{sec:torus} we present the extension of $S$ over the torus $\mathbb T_\infty^2$ and prove Theorem C. 

\vglue 0.3truecm
 
\subsection*{Acknowledges} The authors want to thank Armengol Gasull who point out the works of Bischi {\it et al}, and Arturo Vieiro for helpful comments on previous stages of this work. We also thank the anonymous referees for their valuable comments which substantially improve a previous version of this paper.

\section{The secant map on the real plane: Proof of Theorems A and B}\label{sec:FinitePlane}

The aim of this section is to prove Theorems A and B. The proof of Theorem A splits in several technical lemmata. Remember that $S$ is the secant map defined on (\ref{eq:secant_real}) applied to a real  polynomial of degree $k \geq 2$  
\begin{equation*}
p(x)= a_k x^k + a_{k-1}x^{k-1} + \cdots + a_1 x + a_0 , \ a_k=1,
\end{equation*}
having $1\leq n \leq k$  real roots $\alpha_1<\ldots <\alpha_n$, all simple.  The $n-k\geq 0$ roots are, if any, complex conjugate. Set

\begin{equation}\label{eq:no_def}
\begin{split}
&\delta_{S}^1=\left \{ (x,y) \in \mathbb R^2 \, | \, p(x) = p(y) \, \hbox{ with } x\neq y   \right \},\ \delta_{S}^2=\left  \{(x,x) \in \mathbb R^2 \, | \, p'(x)=0 \right\}, \\
&
\delta_S=\delta_{S}^1 \cup \delta_{S}^2 \quad  {\rm and} \quad  E_S =  \mathbb R^2 \setminus \bigcup_{n \geq 1}^{\infty} S^{-n} \left(\delta_{S}\right).
\end{split}\end{equation}

We remark that $\delta_S$ is a symmetric plane real algebraic curve intersecting the line $y=x$ precisely at points in $\delta^2_S$. According to previous discussion there is an implicit uncertainty on how to define the image at points where the denominator of the second component of $S$ is zero, i.e., where $p(x)-p(y)=0$.  The following lemmas show that indeed $S$ is well defined and smooth on $R^2 \setminus \delta_S$. We define the following auxiliary polynomials 
\begin{equation} \label{eq:q-and-q_j}
\begin{split}
&q_j(x,y):=x^{j-1}  + x^{j-2}y + \cdots xy^{j-2}+y^{j-1}, \, \ j=0,\ldots,n \\
& q(x,y):=\sum_{j=1}^k a_j q_j(x,y).
\end{split}
\end{equation}

\begin{lemma}\label{lem:pol_q}
The symmetric polynomial $q(x,y)$ defined above satisfies 
$$
p(x)-p(y)=(x-y) q(x,y).
$$ 
Moreover 
$q(x,x)=p'(x)$ and $\frac{\partial q}{\partial x}(x,x)= \frac{\partial q}{\partial y}(x,x)= \frac{1}{2}p''(x)$.
\end{lemma}

\begin{proof}
Fix a natural $j \geq 1$. Simple computations show that 
\[
x^j - y^j = (x-y)(x^{j-1}  + x^{j-2}y + \cdots xy^{j-2}+y^{j-1})=(x-y)\ q_j(x,y).
\]
Thus,
\begin{equation}\label{eq:def_q}
p(x)-p(y)=\sum_{j=1}^k a_j \left( x^j - y^j \right) = (x-y) \ \sum_{j=1}^k a_j q_j(x,y) = (x-y) \ q(x,y).
\end{equation}
In other words the factor $(x-y)$ divides $p(x)-p(y)$ and the resultant quotient is a  symmetric polynomial. Moreover, since  $ q_j(x,x)= jx^{j-1}$, we get
\[
 q(x,x)=\sum_{i=1}^k a_j q_j(x,x)=\sum_{j=1}^k a_j j x^{j-1} = p'(x).
 \]
Since  
\[
 \frac{\partial q_j}{\partial x} (x,y) = (j-1) x^{j-2} + (j-2)x^{j-3}y + \cdots + y^{j-2}
\]
we have 
$$ \frac{\partial q_j}{\partial x} (x,x)= \frac{1}{2}j(j-1) x^{j-2} \quad  {\rm and} \quad  \frac{\partial q}{\partial x} (x,x) = \frac{1}{2} \sum_{j=1}^k a_j  j(j-1) x^{j-2} = \frac{1}{2} p''(x).
$$
The result for $ \frac{\partial q}{\partial y} (x,x) $ follows similarly.
\end{proof}

\begin{lemma}\label{lem:singularity}
The secant map defined on (\ref{eq:secant_real}) writes as 
\begin{equation}
\label{eq:def_S_2}
S (x,y) = \left( y, \frac{y q(x,y)-p(y)}{q(x,y)}\right)
\end{equation}
for all  $(x,y)\in \mathbb R^2 \setminus \delta_S$. In particular $S:E_S\to E_S$ defines a smooth dynamical system. Moreover 
\begin{equation}\label{eq:def_DS}
DS\left(x,y\right) =
\left(
\begin{array}{cc}
0 & 1 \\ 
A\left(x,y\right) & B\left(x,y\right)\end{array}\right),
\end{equation}
where
\begin{equation*}
A\left(x,y\right) =\frac{p(y)}{q^2(x,y)}\frac{\partial q}{\partial x}(x,y)\quad {\rm and} \quad 
B\left(x,y\right) = \frac{p(x)}{q^2(x,y)}\frac{\partial q}{\partial y}(x,y). 
\end{equation*}
\end{lemma}

\begin{proof} 
From the previous lemma, we have that
$$
q(x,y)=\frac{p(x)-p(y)}{x-y}, 
$$
is a symmetric polynomial not vanishing outside $\delta_S$. So $S$ is well-defined and smooth map on $\mathbb R^2 \setminus \delta_S$ defining a smooth dynamical system on $E_S$. The differential matrix  $DS(x,y)$ is a direct computation.   
\end{proof}

\begin{remark} \label{remark:q} 
Observe that for a given $(x,y)\in \mathbb R^2$ with $x\ne y$ the value of $q(x,y)$ is the slope of the secant line through the points $(x,p(x))$ and $(y,p(y))$. Moreover $S(x,x)=\left(x,N_p(x)\right)$ where $N_p$ is the Newton method applied to $p$. In particular
\begin{equation}\label{eq:DS_diagonal}
DS(x,x)= 
\left(\begin{array}{cc}
0 & 1 \\
\frac{p(x)p''(x)}{2[p'(x)]^2} & \frac{p(x)p''(x)}{2[p'(x)]^2} \end{array}\right).
\end{equation}
\end{remark}

Next lemma determines the focal points and prefocal lines for the secant map $S$.

\begin{lemma} \label{lem:foc} Let $S$ be the secant map. The following statements hold.
\begin{enumerate}
\item[(a)] Let $Q_{i,j}:=(\alpha_i,\alpha_j)$ for $i,j=1\ldots n, \ i\ne j $. The set of focal points of the secant map is given by 
$$
\cal Q:=\cal Q_S = \bigcup_{\substack{i,j=1,\ldots , n  \\ i\ne j}}   Q_{i,j}  . 
$$
\item[(b)] For a given focal point $Q_{i,j}:=\left(\alpha_i,\alpha_j\right)$, $i\neq j$, the prefocal line is given by $L_{Q_{i,j}}=\{(x,y) \in \mathbb R^2 \ | \ x=\alpha_j \}$. Hence, for a given $j$, the focal points $Q_{i,j}, \ i=1,\ldots n, \ i\ne j$ share the same prefocal line.
\end{enumerate}
\end{lemma}

\begin{proof}
By definition, if $Q=\left(x_0,y_0\right)\in \mathbb R^2$ is a focal points of $S$ then the evaluation of $S_2(Q)$ (where $S_2$ denotes the second component of $S$) takes the form $0/0$. According to (\ref{eq:def_S_2}) there are no focal points on the line $x=y$, since otherwise $p(Q)=p'(Q)=0$, a contradiction with the assumption that $p$ has no multiple real roots.
Therefore, again from (\ref{eq:def_S_2}), focal points should be solutions of 
\[
\left \{
\begin{array}{lrl}
y q(x,y)-p(y) &= & 0, \\
q(x,y)&= & 0, \\
\quad\quad\quad\quad\quad\ \ x &\ne& y.
\end{array}
\right.
\]
If $x_0\ne y_0$ and $q(x_0,y_0)=0$, we conclude that $p(y_0)=0$, which in turns implies (see Lemma \ref{lem:pol_q}) $p(x_0)=0$. Therefore we conclude that $S$ has precisely $n(n-1)$ focal points located at  $Q_{i,j}=(\alpha_i,\alpha_j), \ i,j=1,\ldots n,\ i \neq j$. This proves (a).
 
\noindent From definition the prefocal line of the focal point $Q_{i,j}$, is given by 
$x=S_1\left(Q_{i,j}\right)$ being $S_1$ the first component of $S$. Thus,
$$
L_{Q_{i,j}}= \{(x,y) \in \mathbb R^2 \ | \ x=\alpha_j\},
$$
and the lemma follows. In Figure \ref{fig:focals_secant} we sketch the distribution of focal points and prefocal lines  for  $n=3$.
\end{proof}

Next lemma shows the unboundedness of the attracting basins of the fixed points of $S$.

\begin{lemma}\label{lem:attracting_basins} 
Let $S$ be the secant map. The following statements hold.
\begin{itemize}
\item[(a)] The only fixed points of $S$ are the points 
$\left(\alpha_\ell,\alpha_\ell\right),\ \ell=1,\ldots,n$, and they are all attracting.
\item[(b)] Each basin of attraction $A\left(\alpha_\ell\right),\ \ell=1,\ldots n$, is unbounded.  If $n=1$ or $n=2$, then $A^\star\left(\alpha_1\right)$ (and $A^\star\left(\alpha_2\right)$, if $n=2$) are unbounded. If $n\geq 3$ then $A^\star\left(\alpha_1\right)$ and $A^\star\left(\alpha_n\right)$ are unbounded while $A^\star\left(\alpha_\ell\right)$ with $2 \leq \ell \leq n-1$ are bounded.
\end{itemize}
\end{lemma}

\begin{proof}
From (\ref{eq:def_S_2}) we know that $\left(x_0,y_0\right)$ is a fixed point of $S$ if and only if  $y_0=x_0$ and $p(x_0)=0$. Thus the fixed points of $S$ are of the form 
$\left(\alpha_\ell,\alpha_\ell\right)$, where $p\left(\alpha_\ell\right)=0$ and $p'(\alpha_\ell)\ne 0$, with $\ell=1,\ldots, n$.  Moreover, from (\ref{eq:DS_diagonal}) we have 
$$
DS\left(\alpha_\ell,\alpha_\ell\right)=\left(\begin{array}{cc}
0 & 1 \\
0 & 0\end{array}\right), \ \ell=1,\ldots,n .
$$
So the two eigenvalues of $DS\left(\alpha_\ell,\alpha_\ell\right)$ are equal to  0, proving thus that fixed point $(\alpha_\ell,\alpha_\ell)$ are all attracting.  This proves statement (a). 

Fix now $\ell =1,\ldots n$. Set 
$$r_H:=\{(x,y)\in \mathbb R^2 \ | \ y=\alpha_\ell\} \ {\rm  and} \ r_V:=\{(x,y)\in \mathbb R^2 \ | \ x=\alpha_\ell\},
$$ 
the horizontal and vertical lines passing through the point $\left(\alpha_\ell,\alpha_\ell\right)$, respectively. It is easy to see from (\ref{eq:def_S_2}) that, on the one hand,  if  $x\ne \alpha_j$ then $S\left(x,\alpha_\ell\right)=(\alpha_\ell,\alpha_\ell)$; and, on the other hand, if  $y\ne \alpha_j$ then $S\left(\alpha_\ell,y\right)=\left(y,\alpha_\ell\right)$, $j=1,\ldots, n,\ j\ne \ell$. This implies that $A\left(\alpha_\ell\right)$ is unbounded since
$$
R_H:=r_H\setminus \bigcup_{j\ne \ell} Q_{j,\ell} \quad  {\rm and}\quad R_V:=r_V\setminus \bigcup_{j\ne \ell} Q_{\ell,j}
$$ 
belong to $A\left(\alpha_\ell\right)$. This prove the first assertion of statement (b).
 
The cases when $n=1$ and $n=2$ are straightforward. So we assume $n\geq 3$. Let  $ \ell =2,\ldots, n-1$. We claim that  $A^*\left(\alpha_\ell\right) \subset R_\ell$ where $R_\ell$ is the  rectangle  with vertices at the points $\left \{ \left(\alpha_{\ell-1},\alpha_{\ell-1}\right), \left(\alpha_{\ell+1},\alpha_{\ell-1}\right),\left(\alpha_{\ell+1},\alpha_{\ell+1}\right), \left(\alpha_{\ell+1},\alpha_{\ell-1}\right)   \right \}$. The claim follows since, according to the arguments above,  $\partial R_{\ell} \subset A\left(\alpha_{\ell-1}\right)\cup A\left(\alpha_{\ell+1}\right)$. If $\ell=1$ or $\ell=n$, then an unbounded piece of the line $r_{H}$ belongs to $A^*(\alpha_{\ell})$ proving that it is unbounded. See Figure \ref{fig:focals_secant} where different colours illustrate points on the basins of attraction of the three fixed points. This finish the proof of the lemma
\end{proof}

\begin{figure}[ht]
    \centering
     \includegraphics[width=0.45\textwidth]{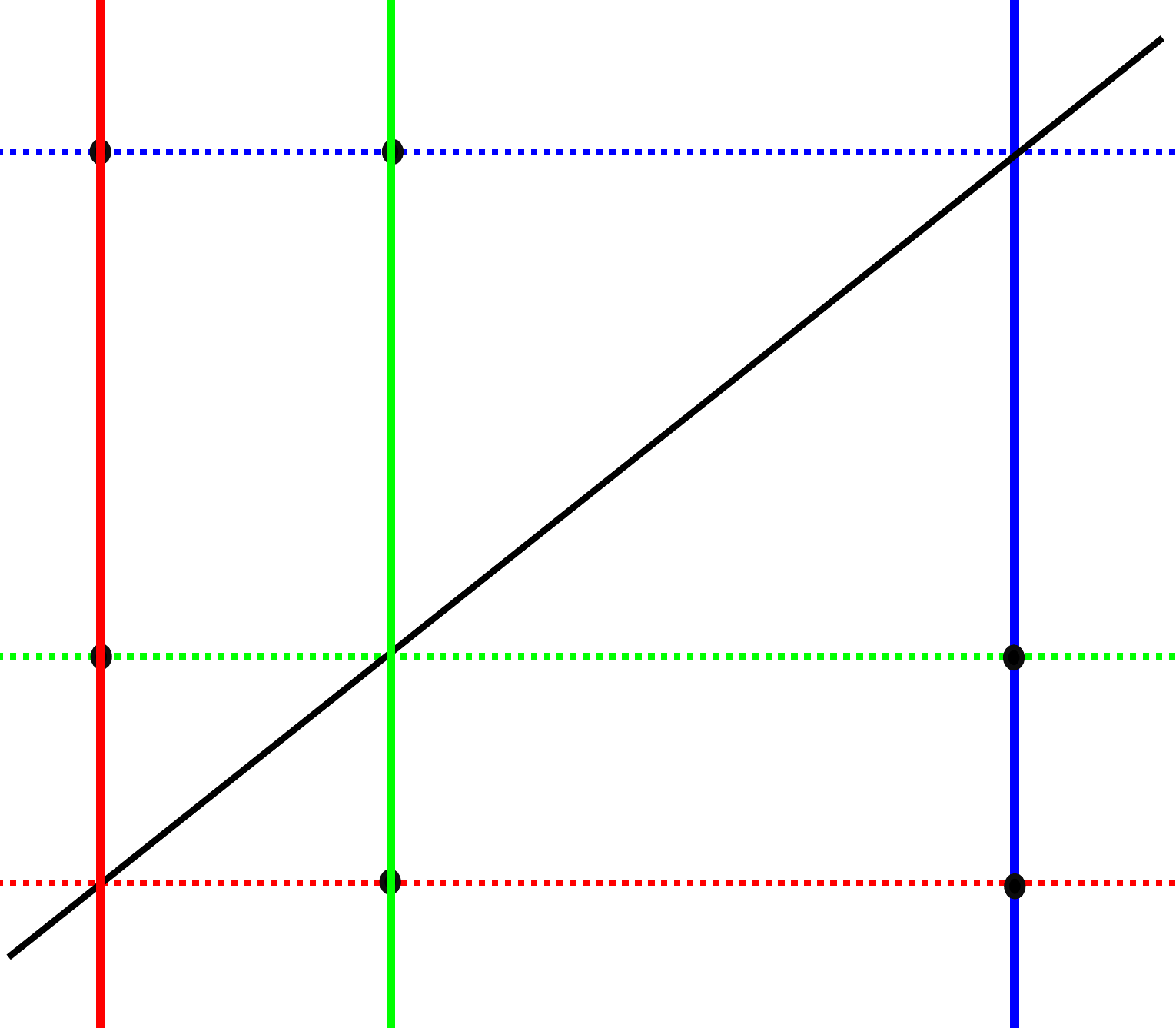}
   \put(-205,55) {\small $Q_{1,2}$ } 
      \put(-205,155) {\small $Q_{1,3}$ } 
        \put(-130,155) {\small $Q_{2,3}$ } 
       \put(-25,55) {\small $Q_{3,2}$ } 
        \put(-130,15) {\small $Q_{2,1}$ } 
         \put(-25,15) {\small $Q_{3,1}$ } 
          \put(-25,137) {\small $(\alpha_3,\alpha_3)$ } 
                    \put(-1,167) {\small $y=x$ } 
        \put(-180,15) {\small $(\alpha_1,\alpha_1)$ } 
        \put(-130,55) {\small $(\alpha_2,\alpha_2)$ } 
        \put(-190,-5) {\small $x=\alpha_1$}
        \put(-140,-5) {\small $x=\alpha_2$}
        \put(-40,-5) {\small $x=\alpha_3$}
           \caption{\small{Sketch of the dynamical plane of $S_p$ where $p$ is a polynomial with three simple real roots $\alpha_1<\alpha_2<\alpha_3$. The focal focal points $Q_{2,1}$ and $Q_{3,1}$ share the prefocal line
           $x=\alpha_1$. The focal points $Q_{1,2}$ and $Q_{3,2}$ share the prefocal line $x=\alpha_2$, and finally, the focal points $Q_{1,3}$ and $Q_{2,3}$ share the prefocal line $x=\alpha_3$. Red points are seeds converging to $(\alpha_1,\alpha_1)$, green points converge to $(\alpha_2,\alpha_2)$ and blue points converge to $(\alpha_3,\alpha_3)$, as shown in Lemma \ref{lem:attracting_basins}.}}
    \label{fig:focals_secant}
    \end{figure}
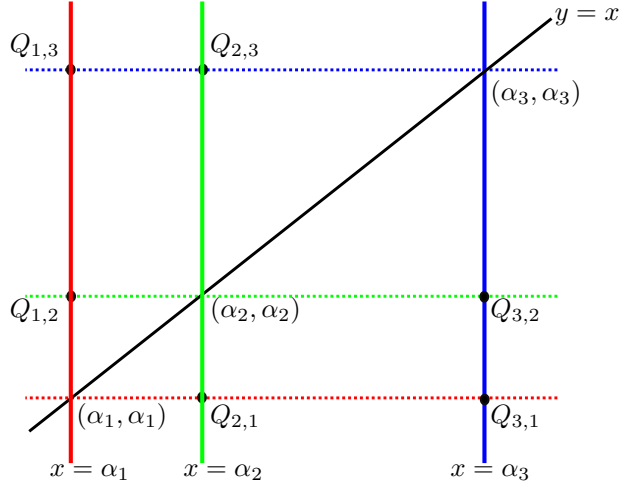

The following consequence of H\^opital's rule will be needed in the proof of statement (d) of Theorem A.
 
\begin{lemma}\label{lemma:hopital}
Let $f,g:\left(-\varepsilon,\varepsilon\right) \to \mathbb R$ be two smooth functions such that for some $a\in \mathbb R$ and $b\ne 0$ we have
\[
\lim_{t \to 0} f(t)=\lim_{t \to 0} g(t)=0,\quad \lim_{t \to 0} f'(t)=a \quad {\rm and} \quad \lim_{t \to 0} g'(t)=b.
\]
Then
\[
\lim_{t \to 0} \left( \frac{f(t)}{g(t)} \right)' = \frac{1}{2b}  \left( f''(0)-g''(0)\frac{a}{b} \right).
\]

 \end{lemma} 
 
\begin{proof}
Since
\[
\lim_{t \to 0} \left( \frac{f(t)}{g(t)} \right)'  = \lim_{t \to 0} \frac{f'(t)g(t)-g'(t)f(t)}{g^2(t)} = \frac{0}{0},
\]
we apply H\^opital's rule to obtain
\[ 
\lim_{t \to 0} \left( \frac{f(t)}{g(\tau)} \right)' = \lim_{t \to 0} \left( \frac{f''(t)}{2 g'(t)} - \frac{g''(t)f(t)}{2 g'(t)g(t)}  \right) = \frac{1}{2b} \left( f''(0)-g''(0)\frac{a}{b} \right).
\]
\end{proof} 

\vglue 0.3truecm 
\noindent{\it Proof of Theorem A.} A major part of the proof follows from previous lemmas. From Lemma \ref{lem:singularity} we know that the dynamical system $S:E_S \to E_S$ is smooth. Statements (a) and (b) follows from Lemma \ref{lem:attracting_basins}. Statement (c) follows from Lemma \ref{lem:foc}. So, to finish the proof of Theorem A we deal with statement (d). 

Fix $Q:=Q_{i,j}=\left(\alpha_i,\alpha_j\right)$, a focal point (so $i\ne j$) and let $U:=U_{i,j}$ be a sufficiently small punctured neigborhood of $Q$ (in particular $U$ does not intersect other focal points). See Figure \ref{fig:focals_secant_2}. We know from previous arguments (see proof of Lemma \ref{lem:attracting_basins}) that the segment of $x=\alpha_i$ in $U$ belongs to $A\left(\alpha_i\right)$ and that the segment of $y=\alpha_j$ in $U$ belongs to $A\left(\alpha_j\right)$, thus it follows that $Q\in \partial A\left(\alpha_j\right)\cap \partial A\left(\alpha_i\right)$.
To finish the proof we need to show that $Q\in \partial A\left(\alpha_\ell\right)$ for all $\ell \ne i,j$. 

Putting together (\ref{eq:PlaneDenominator}) and (\ref{eq:def_S_2})  we have
\begin{equation}\label{eq:N_i_D}
F(x,y)=y, \, N(x,y)= y q(x,y)-p(y), \ D(x,y)=q(x,y).
\end{equation}
Moreover from Lemma \ref{lem:pol_q} we also have 
\begin{equation}\label{eq:nxnydxdy}
N_x(Q) = \frac{\alpha_j \ p'(\alpha_i)}{\alpha_i-\alpha_j}, \ N_y(Q)= \frac{-\alpha_i\ p'(\alpha_j)}{\alpha_i-\alpha_j}, \ D_x(Q)=\frac{p'(\alpha_i)}{\alpha_i-\alpha_j}  \ {\rm and} \ D_y(Q)=-\frac{p'(\alpha_j)}{\alpha_i-\alpha_j}.
\end{equation}
So, we conclude
\begin{equation}\label{eq:theorem1.1_applies}
N_x(Q)D_y(Q) - N_y(Q) D_x(Q) = \frac{p'(\alpha_i) p'(\alpha_j)}{\alpha_i-\alpha_j} \ne 0.
\end{equation}

Let $\gamma_m=\gamma_m(t),\ t\in\left(-\varepsilon,\varepsilon\right)$ be a curve passing through $Q$ (at $t=0$) with slope $m:=\gamma_m^{\prime}(0)$  not tangent to $\delta_S$.
From (\ref{eq:theorem1.1_applies}) and Theorem \ref{th:Gardini}  we know there is a one-to-one correspondence between $m \in \mathbb R \cup \{\infty\}$ and the points of the prefocal line $L_Q=\{(x,y)\in \mathbb R^2 \ | \ x=\alpha_j \}$. Moreover, from (\ref{eq:y(m)}) and (\ref{eq:nxnydxdy}) we conclude that the bijection is given by 
\begin{equation}\label{eq:y(m)_Sp}
y(m)=\frac{\alpha_j p'(\alpha_i)-\alpha_i p'(\alpha_j)m}{p'(\alpha_i)-p'(\alpha_j)m} \quad {\rm or} \quad 
m(y)= \frac{p'\left(\alpha_i\right)\left(\alpha_j-y\right)}{p'\left(\alpha_j\right)\left(\alpha_i-y\right)},
\end{equation}
and so if $m\ne m_\ell,\ \ell = 1, \ldots, n$, with  
\begin{equation}\label{eq:singular_slopes}
m_{\ell}= \frac{p'(\alpha_i)}{p'(\alpha_j)}  \frac{\left(\alpha_j-\alpha_{\ell}\right)}{\left(\alpha_i-\alpha_{\ell}\right)}, 
\end{equation}
the curve $S\left(\gamma_m\right)$ crosses the prefocal line $L_Q$ through a point $(\alpha_j,y)$ not being a focal point (a blue point in Figure \ref{fig:focals_secant_2}). Hence, shrinking  $U$ is necessary, $\gamma_m \cap U\in A\left(\alpha_j\right)$.

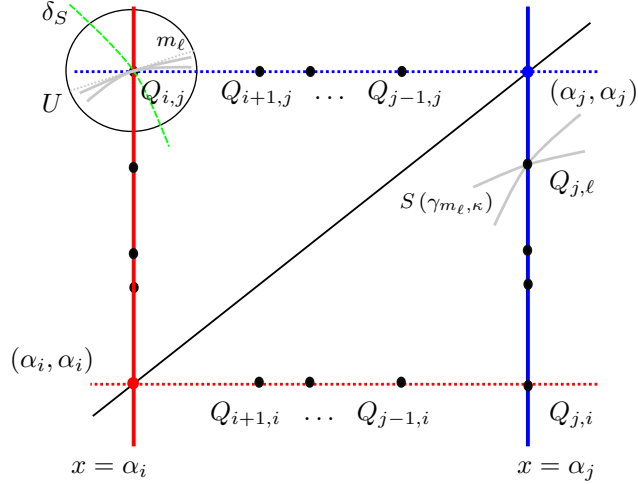
\begin{figure}[ht]
    \centering
     \includegraphics[width=0.45\textwidth]{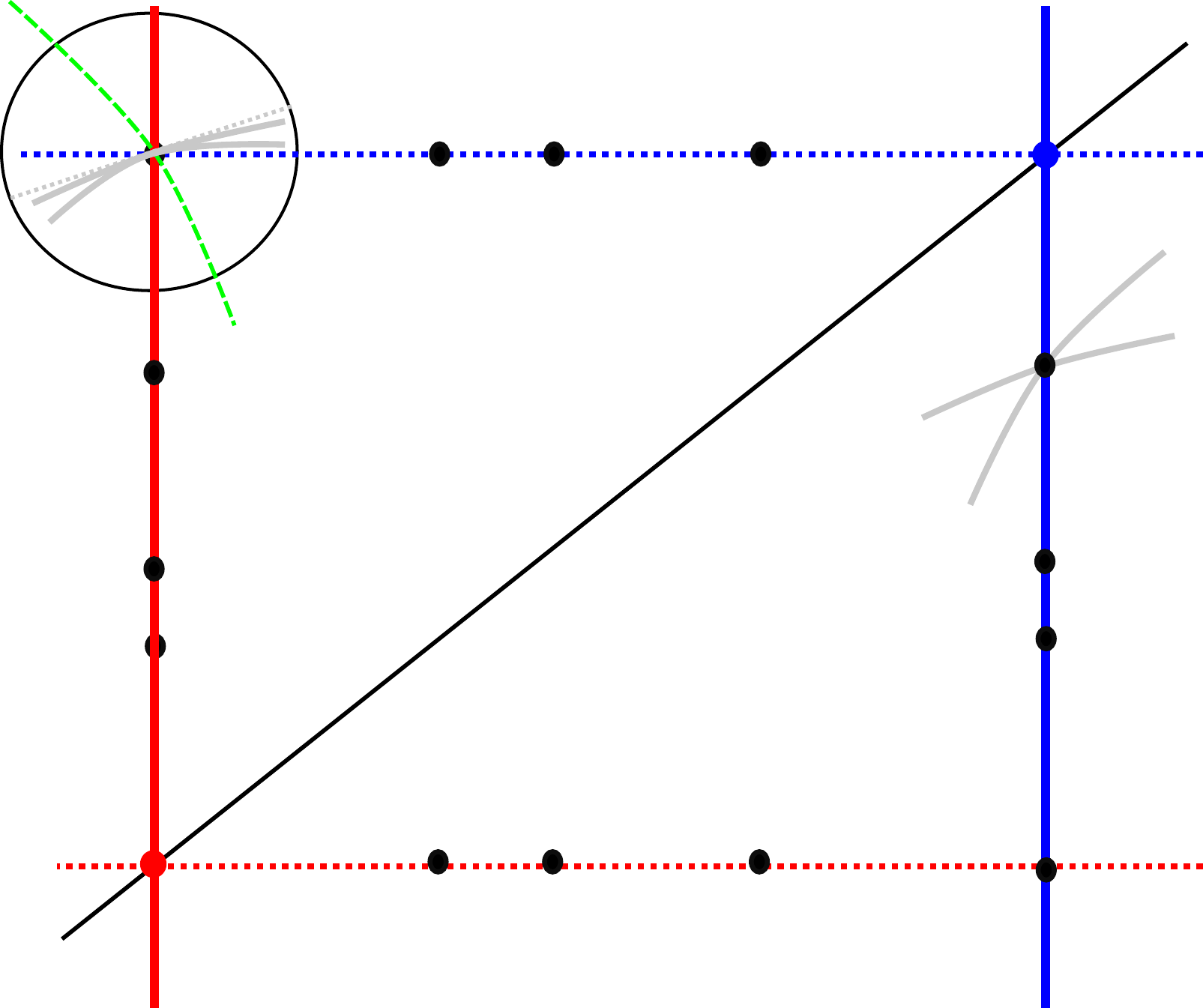}
    \put(-208,161) {\small $\delta_S$}
       \put(-208,126) {\small $U$}
    \put(-171,131) {\small $Q_{i,j}$}
       \put(-165,151) {\scriptsize $m_\ell$}
   \put(-141,131) {\small $Q_{i+1,j}$} 
    \put(-106,131) {\small $\ldots$} 
     \put(-86,131) {\small $Q_{j-1,j}$} 
      \put(-18,131) {\small $(\alpha_j,\alpha_j)$}
         \put(-18,97) {\small $Q_{j,\ell}$} 
         \put(-74,90) {\scriptsize $S\left(\gamma_{m_{\ell},\kappa}\right)$}
         \put(-220,30) {\small $(\alpha_i,\alpha_i)$}
   \put(-145,10) {\small $Q_{i+1,i}$} 
    \put(-110,10) {\small $\ldots$} 
     \put(-90,10) {\small $Q_{j-1,i}$} 
      \put(-18,10) {\small $Q_{j,i}$} 
      \put(-197,-10) {\small $x=\alpha_i$} 
      \put(-30,-10) {\small $x=\alpha_j$}
             \caption{\small{Sketch of the proof of Theorem A.}}
    \label{fig:focals_secant_2}
    \end{figure}

According to the previous paragraph we fix in what follows $\ell \ne i,j$, and consider the family of curves depending on the parameter $\kappa\in \mathbb R$ given by
\[
\gamma_{m_{\ell},\kappa}(t)=:(x(t),y(t))= \left(\alpha_i,\alpha_j\right) + (1,m_\ell) t + \frac{1}{2}(1,\kappa) t^2 = \left(\alpha_i+ t + \frac{1}{2} t^2, \, \alpha_j+m_\ell t + \frac{1}{2}\kappa t^2\right).
\]
We notice that all curves in the family pass through $Q$ with {\it singular} slope $m_\ell,\ \ell\ne i,j$. That is, $S\left(\gamma_{m_\ell,\kappa}\right)$ cross the prefocal line $L_Q$ at a focal point $Q_{j,\ell}$. The parameter $\kappa$ defines the curvature of the curve $\gamma_{m_\ell,\kappa}$ when passing through $Q$. 

We claim that there exists a one-to-one correspondence between the parameter $\kappa\in \mathbb R$ and $\left(S\circ \gamma_{m_{\ell},\kappa}\right)^{\prime}(0)$, or equivalently, we claim that choosing different values of $\kappa\in \mathbb R$ the curves $S\left(\gamma_{m_{\ell},\kappa}\right)$ pass through $Q_{j,\ell}$  with all possible slopes $m\in \mathbb R$. See Figure \ref{fig:focals_secant_2}.
Assuming the claim is true this would imply that applying $S$ once more, i.e. $S^2\left(\gamma_{m_{\ell},\kappa}\right)$, we will get curves passing through all points on the prefocal line $L_{Q_{j,\ell}}$, in particular passing through the point $\left(\alpha_\ell,\alpha_\ell\right)\in A\left(\alpha_\ell\right)$, and so we would conclude that $Q_{i,j}\in \partial  A\left(\alpha_\ell\right)$. Since this argument works for all $\ell \ne i,j$ we would have the desired result. 

Now we prove the claim. Observe that
$$
\left(S\circ \gamma_{m_{\ell},\kappa}\right)^{\prime}(0)=\left(m_\ell,\left(\frac{f(t)}{g(t)}\right)^{\prime}|_{t=0}\right),
$$
where $f(t):=N\left(x(t),y(t)\right)$ and $g(t):=D\left(x(t),y(t)\right)$, and $N$ and $D$ (numerator and denominator of the second component of $S$) are written explicitly on (\ref{eq:N_i_D}). Since $Q$ is a focal point we have
$$
\lim_{t\to 0} \frac{f(t)}{g(t)}=\frac{0}{0}.
$$
Some computations show that
$$
\lim_{t \to 0} f^{\prime}(t) = \frac{\alpha_jp^{\prime}\left(\alpha_i\right)-\alpha_ip^{\prime}\left(\alpha_j\right)m_{\ell}}{\alpha_i-\alpha_j}:=a \quad {\rm and} \quad \lim_{t \to 0} g^{\prime}(t) = \frac{p^{\prime}\left(\alpha_i\right)-p^{\prime}\left(\alpha_j\right)m_{\ell}}{\alpha_i-\alpha_j}:=b.
$$
We claim that $b\ne 0$. Indeed, otherwise, $y\left(m_\ell\right)=\infty$ while the slope $m_\ell$ corresponds to focal points of the form
$\left(\alpha_j,\alpha_\ell\right),\ \ell \ne j$. Hence, we are on the hypothesis of Lemma \ref{lemma:hopital} to get
\begin{equation} \label{eq:limit_kappa}
\lim_{\tau \to 0} \left(\frac{f}{g} \right)'(\tau)=\frac{1}{2b} \left( f''(0)-g''(0)\frac{a}{b} \right).
\end{equation}
Finally some computations show that

\[
\begin{array}{ll}
f''(t)= & (x'(t),y'(t)) \, H(N)(x(t),y(t)) \, (x'(t),y'(t))^T \, + \, \nabla (N) (x(t),y(t)) \, (x''(t),y''(t))^T \\
g''(t)= & (x'(t),y'(t)) \, H(D)(x(t),y(t)) \, (x'(t),y'(t))^T \, + \, \nabla (D) (x(t),y(t)) \, (x''(t),y''(t))^T 
\end{array}
\]
\noindent  where $H$ denotes the Hessian matrix.

Thus, 
\begin{equation}
\begin{split}
f''(0)&=  N_{xx}(Q)  + 2 N_{xy}(Q)  m_\ell  +  N_{yy}(Q)  m_\ell^2  + N_x(Q) +  N_y(Q) \kappa  \\
g''(0)&= D_{xx}(Q)  + 2 D_{xy}(Q)  m_\ell  +  D_{yy}(Q)  m_\ell^2  + D_x(Q) +  D_y(Q)  \kappa
\end{split}
\end{equation}
Substituting on the right hand side expression of (\ref{eq:limit_kappa}) we see that the $\kappa$-coefficient is given by 
$$
N_y(Q)-D_y(Q)\frac{a}{b}=N_y(Q)-D_y(Q) \alpha_{\ell}=\frac{\alpha_{\ell}-\alpha_i}{\alpha_i-\alpha_j}p'\left(\alpha_j\right)\ne 0,
$$ 
as desired. Hence statement (d) follows. \qed

In Figure \ref{fig:dyn_plane}(a)-(b) we illustrate Theorem A for a concrete polynomial of degree three with three real roots. Each root $\left(\alpha_j,\alpha_j\right),\ j=1,2,3$ has a basin of attraction associated to a different colour (red, green and blue, respectively). Fix the attention to the focal point $Q_{3,2}$ and its prefocal line given by the vertical line $x=\alpha_2$. In  Figure \ref{fig:dyn_plane}(b) we can see that in a neighbourhood of $Q_{3,2}$ all curves $\gamma:=\gamma(t),\ t\in(-\varepsilon,\varepsilon)$ passing through $Q_{3,2}$ with slope $m\ne \{m_1\approx (2\pi)/3,\infty\}$, are coloured in  green. This is so because if $m\ne \{m_1,\infty\}$, $S\left(\gamma\right)$ is  a curve passing through a point $\left(\alpha_2,y\right)\in L_{Q_{3,2}}$ with 
$y\ne \alpha_\ell$, $\ell=1,3$ and so $\gamma\in A(\alpha_2)$, if $\varepsilon$ is small enough  (see Figure \ref{fig:dyn_plane}(a)).  The situation is quite different if $m=m_1$ and $m=\infty$. For instance suppose that
$\gamma$ is a curve passing through $Q_{3,2}$ (at $t=0$) with slope $m=\infty$ (i.e., $\gamma'(0)=\infty$). From the arguments in the proof of Theorem A we know that its image is a curve passing through $Q_{2,3}\in L_{Q_{3,2}}$ with all possible slopes depending on $\gamma^{\prime\prime}(0)$. Since in turn $Q_{2,3}\in \left( \partial A\left(\alpha_1\right) \cap \partial A\left(\alpha_2\right) \cap \partial A\left(\alpha_3\right)\right)$ implies that depending on $\gamma^{\prime\prime}(0)$ we see blue, red and green near $Q_{3,2}$ when passing through this point with slope $m=\infty$. Moreover since all curves passing through the point $Q_{2,3}$ belong to the basin of attraction of $\left(\alpha_3,\alpha_3\right)$ with at most two exceptions (that is, a neighbourhood of $Q_{2,3}$ is generically blue) we see that most of the curves passing through $Q_{3,2}$ with slope $m=\infty$ are blue.

\begin{figure}[ht]
    \centering
    \subfigure[\scriptsize{Range [-1,5]x[-1,5].}  ]{
     \includegraphics[width=0.48\textwidth]{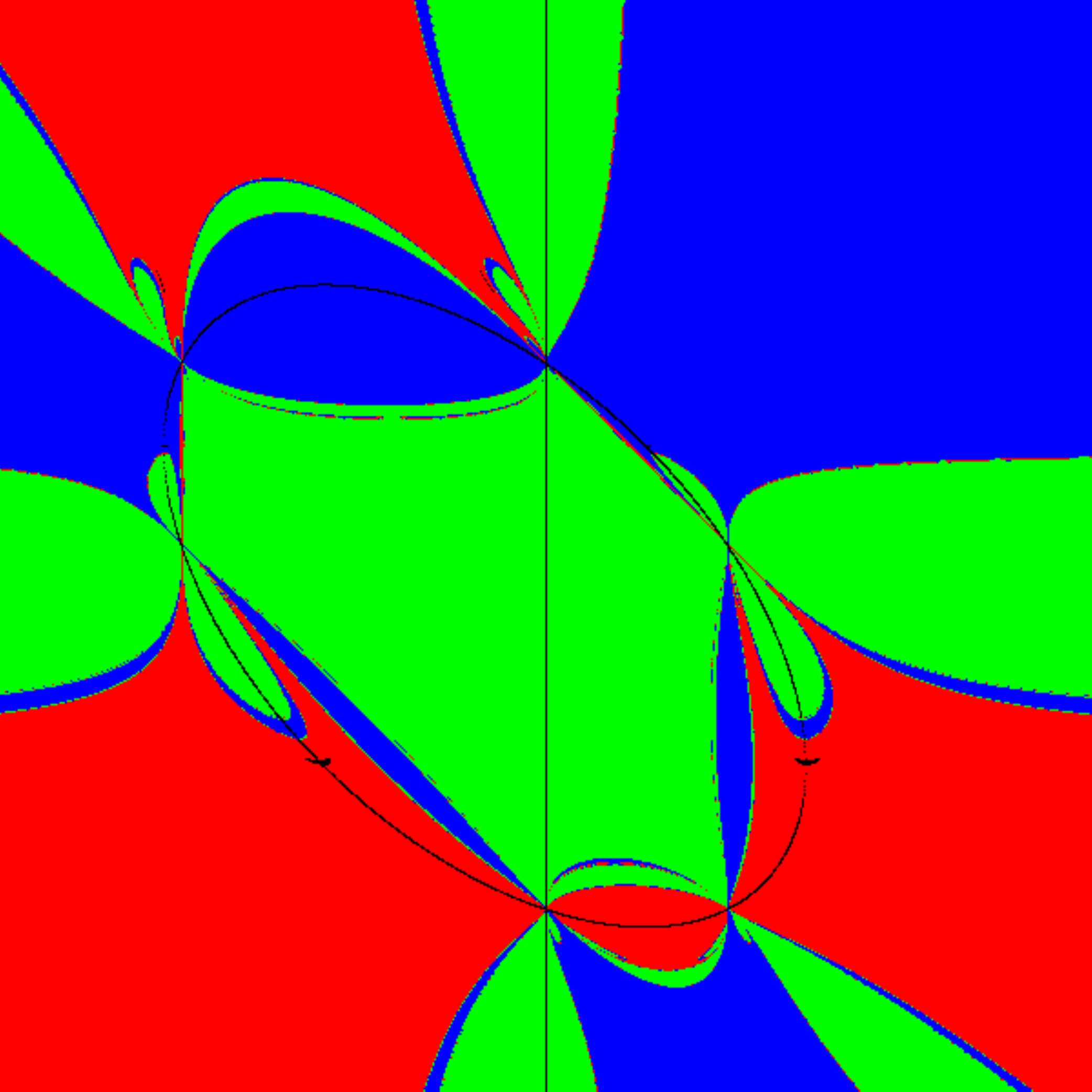}
    \put(-65,105) {\small $Q_{3,2}$ } 
     \put(-102,140) {\small {\w $Q_{2,3}$} }
      \put(-142,190) {\small $x=\alpha_2$} 
    \put(-107,105) {\circle{4}}
    \put(-143,103) {\small $(\alpha_2,\alpha_2)$}
     \put(-167,46) {\circle{4}}
         \put(-164,40) {\small $(\alpha_1,\alpha_1)$}
     \put(-67,146) {\circle{4}}
         \put(-63,143) {\small $(\alpha_3,\alpha_3)$}
      }
    \subfigure[\scriptsize{Range [2.34,3.85]x[1.28,2.79].}  ]{
     \includegraphics[width=0.48\textwidth]{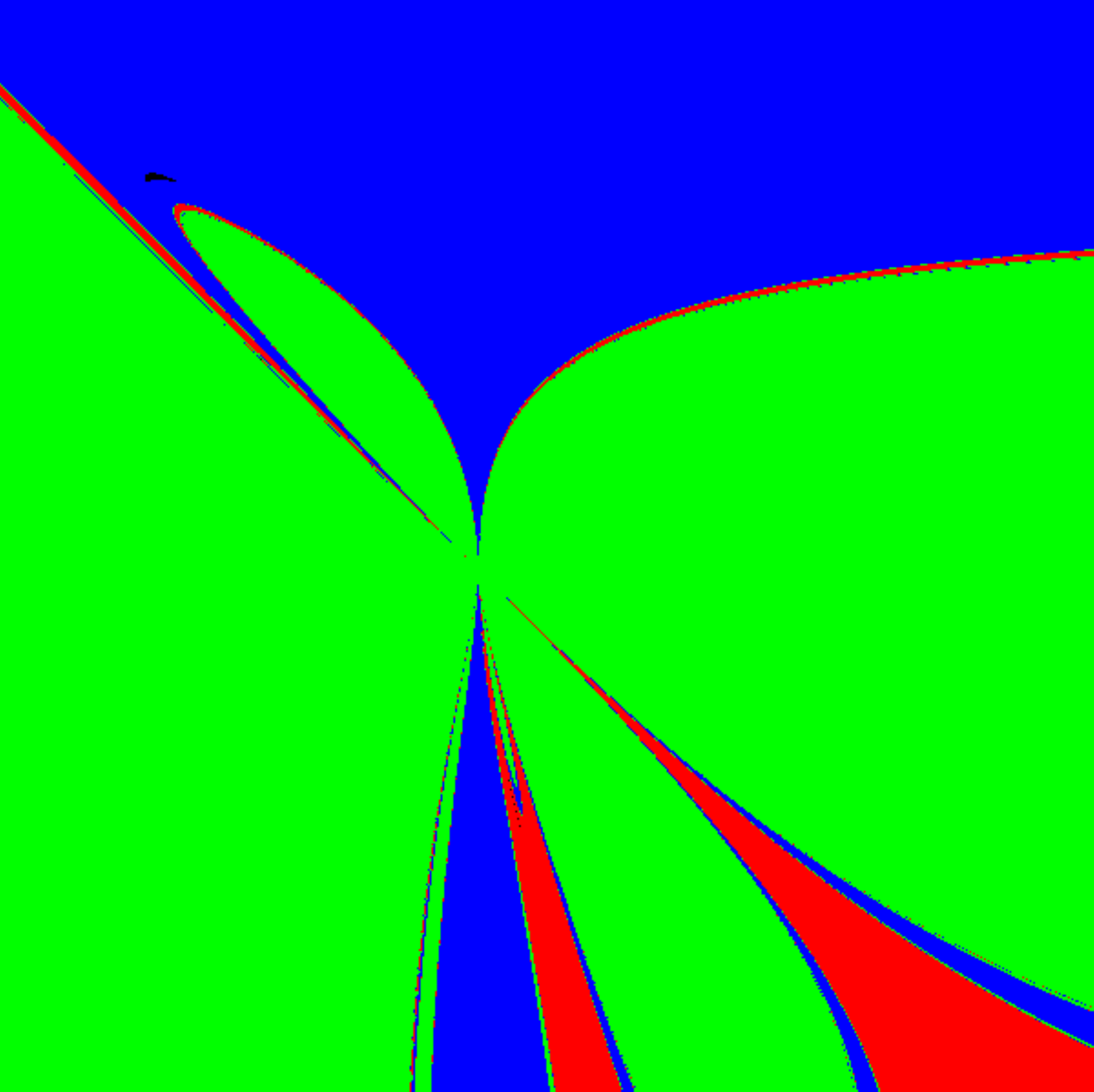}
    \put(-110,98) {\small $Q_{3,2}$ } 
 
      }   
           \caption{\small{(a) The dynamical plane of $S$ applied to the polynomial $p(x)=x(x-2)(x-3)$. We show in red the basin of attraction of $\left(\alpha_1,\alpha_1\right)=(0,0)$, in green the basin of attraction of $\left(\alpha_2,\alpha_2\right)=(2,2)$ and in blue the basin of attraction of $\left(\alpha_3,\alpha_3\right)=(3,3)$. We show the set $\delta_S$ where the map $S$ is not well defined. We also plot the prefocal line $x=\alpha_2$  associated to the focal points $Q_{3,2}$ and $Q_{1,2}$. (b) Zoom in the dynamical plane around the focal point $Q_{2,3}$. It is also shown some small moon-shaped black regions. We will explain their meaning in next section.}}
    \label{fig:dyn_plane}
    \end{figure}
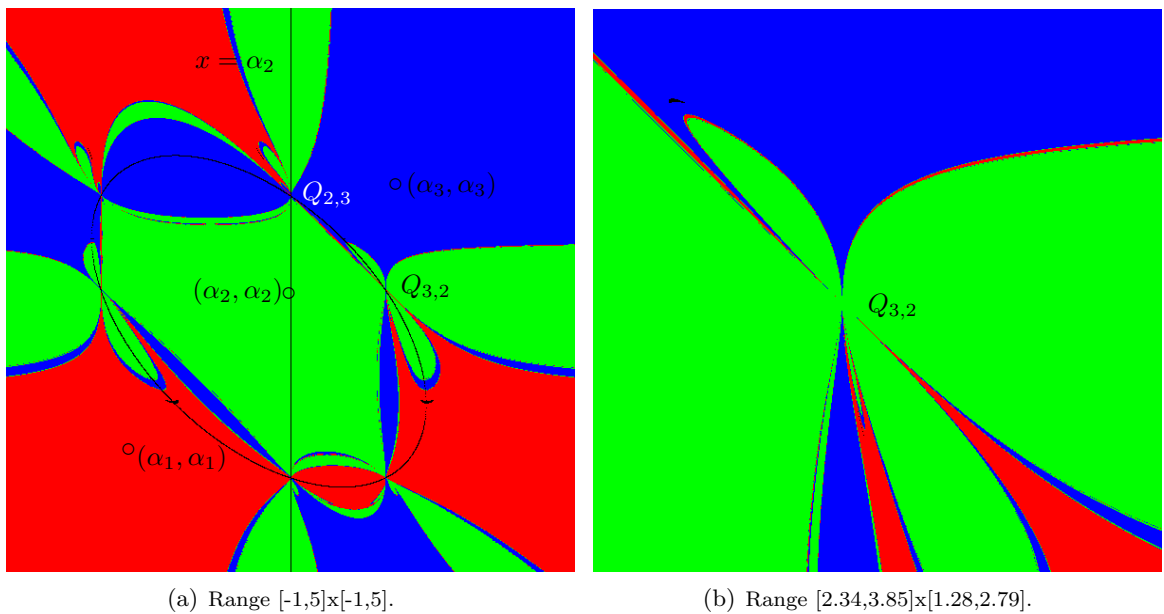

\subsection{Proof of Theorem B}

We show statement (a) by contradiction. We first assume the existence of a periodic orbit of minimal period 2 in $E$, that is  $S(a,b)=(c,d)$ and $S(c,d)=(a,b)$ for some $a,b,c,d\in \mathbb R$ such that $(a,b)$ and $(c,d)$ are in $E$. From (\ref{eq:def_S_2}) we conclude that $c=b$ and $d=a$. So, $S(a,b)=(b,a)$ and $S(b,a)=(a,b)$ with $a \neq b$ (otherwise we would have a fixed point). From (\ref{eq:def_S_2}),  if $S(a,b)=(b,a)$ we conclude that  
\[
a = b - p(b)\frac{b-a}{p(b)-p(a)}.
\]
Notice that $p(a)\ne p(b)$ since $(a,b)\in E$. The above equation writes as 
\[
0= (b-a)\left[ 1- \frac{p(b)}{p(b)-p(a)}\right]=p(a) \frac{(b-a)}{p(b)-p(a)}.
\] 
Since $a\ne b$ the above equation concludes $p(a)=0$. Interchanging the role of $a$ and $b$ we also conclude $p(b)=0$. All together this implies $(a,b)\not\in E$.

We secondly assume the existence of a periodic orbit of minimal period 3 in $E$, that is $S(a,b)=(c,d)$, $S(c,d)=(e,f)$ and $S(e,f)=(a,b)$ for some $a,b,c,d,e,f\in \mathbb R$ such that $(a,b)$, $(c,d)$ and $(e,f)$ are in $E$. Arguing in a similar way as before we have that $c=b$, $e=d$ and $f=a$, so we have that $S(a,b)=(b,d)$, $S(b,d)=(d,a)$ and $S(d,a)=(a,b)$. Since the minimal period of the orbit is three we conclude that the three real numbers $a,b$ and $d$ are different. 

Without loss of generality we assume $a < b < d$ (otherwise we rename the letters). Since the secant line through $(a,p(a))$ and $(d,p(d))$ should cut the line $y=0$ at the point $x=b$ (observe that $S(d,a)=(a,b)$) we know that $p(a)p(d)<0$. Assume $p(a)>0$ and $p(d)<0$ (the other case is similar). This force $p(b)>0$, since the secant line passing through $(a,p(a))$ and $(b,p(b))$ should intersect the line $y=0$ at $x=d$ (observe that $S(a,b)=(b,d)$). Accordingly the secant line through 
$(b,p(b))$ and $(d,p(d))$ will intersect the line $y=0$ at a point $\eta \in (b,d)$, a contradiction with $S(b,d)=(d,a)$ and $a<b$. This finish the proof of statement (a).

Now we deal with statement (b) by showing of the existence of (attracting) periodic $S$-orbits of minimal period 4.  We denote by $a,b,c$ and $d$ four real numbers such that $a < b < c < d$. 

Arguing in a similar way as we did above, and after relabelling the real numbers involved in the construction of the four periodic orbit, the configuration should be as follows (see Figure \ref{fig:period4})
\begin{equation}\label{eq:secant_period4}
\begin{split}
&a<b<c<d \quad {\rm and}\\
&S(a,b)= (b,d)\quad S(b,d)= (d,c)\quad S(d,c)=(c,a)\quad S(c,a)= (a,b).
\end{split}
\end{equation}

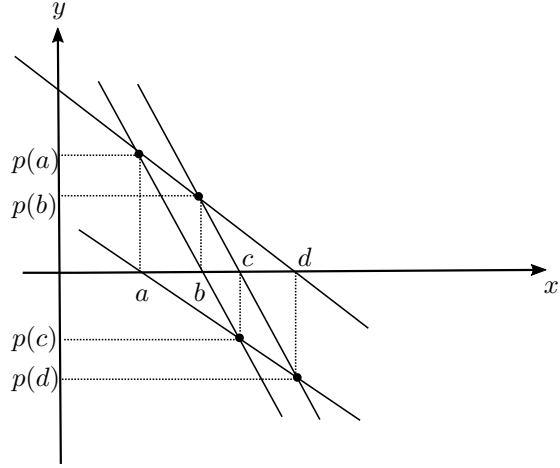
\begin{figure}[ht]
    \centering
     \includegraphics[width=0.45\textwidth]{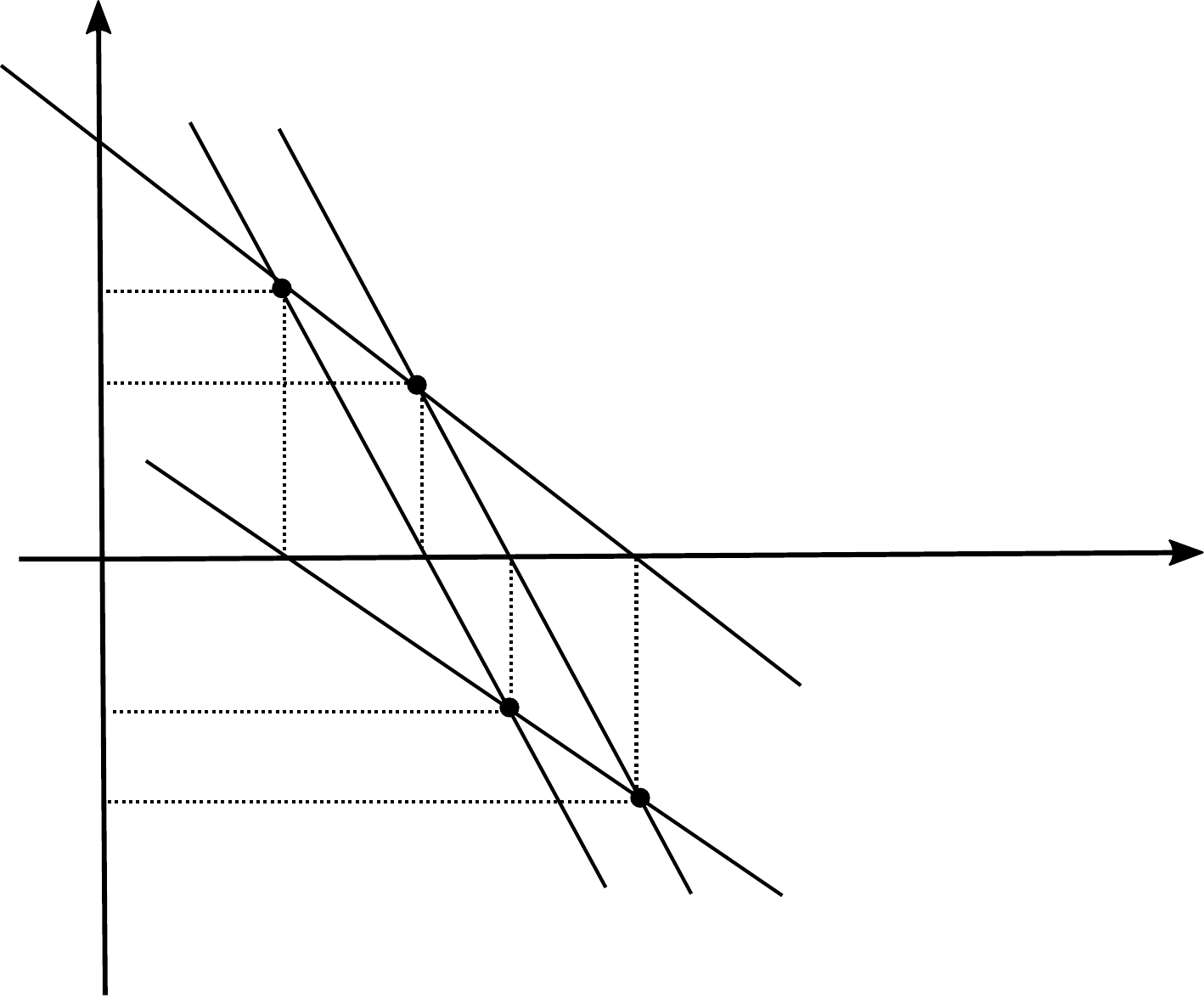}
    \put(-154,62) {\small $a$ } 
     \put(-200,112){\small $p(a)$}
     \put(-200,95){\small $p(b)$}
    \put(-200,45){\small $p(c)$}
   \put(-200,30){\small $p(d)$}
     \put(-114,75) {\small $c$} 
     \put(-1,65) {\small $x$} 
     \put(-185,170){\small $y$}
     \put(-93,75){\small $d$}
     \put(-132,62){\small $b$}
           \caption{\small{Configuration of the period 4 cycle.}}
    \label{fig:period4}
    \end{figure}

To simplify the construction we assume that the secant lines passing through $(a,p(a))$ and $(b,p(b))$, and through $(c,p(c))$ and $(d,p(d))$ have slope equal to $-1$ (observe that this is equivalent to assume $q(a,b)=q(d,c)=-1$). Under this assumption and the fact that  $S(a,b)= (b,d)$ and $S(d,c)=(c,a)$ we get
\begin{equation}\label{eq:cond_period4_1}
p(a)=d-a >0,\quad  p(b)=d-b>0,\quad p(c)= a-c<0 \quad {\rm and} \quad p(d)=a-d<0.
\end{equation}

Of course the inequalities in (\ref{eq:cond_period4_1}) are not enough to fulfil  (\ref{eq:secant_period4}). We should further impose that the secant line passing through $(b,p(b))$ and $(d,p(d))$ crosses the line $y=0$ at the point $x=c$, and that the line passing through $(a,p(a))$ and $(c,p(c))$ crosses the line $y=0$ at the point $x=b$. Easy computations show that these two conditions write as
\begin{equation}
\label{eq:cond_period4_2}
c= d - \frac{(a-d)(d-b)}{a+b-2d} \quad {\rm and} \quad b = a - \frac{(d-a)(a-c)}{c+d-2a}.
\end{equation}

Again doing some straightforward computations one can see that the following (approximate) parameters
$$
\bar{a}=1\ <\ \bar{b}=2\ <\ \bar{c}=\frac{1}{2}\left(3 + \sqrt{5}\right)\approx 2.618\ <\ \bar{d}=\frac{1}{2}\left(5 + \sqrt{5}\right)\approx 3.618
$$
determine a unique interpolating polynomial 
$$
p_{\bar{a},\bar{b},\bar{c},\bar{d}}(x)=2.61803 - (x-1) - 2.61803 (x-1)(x-2) + 2 (x-1)(x-2)(x-2.61803) 
$$ 
satisfying (\ref{eq:cond_period4_1}). Moreover, by construction, $S:=S_{\bar{p}}$ has a four periodic orbit at the points  
$$
\{\left(\bar{a},\bar{b}\right),\left(\bar{b},\bar{d}\right),\left(\bar{d},\bar{c}\right),\left(\bar{c},\bar{a}\right)\}.
$$

Observe that the arguments used above implicitly provide a huge family of polynomials for which the secant method has a four periodic orbit. Our aim is to find one for which the four periodic orbit is attracting. The strategy will be to keep the parameters $\bar{a}<\bar{b}<\bar{c}<\bar{d}$ satisfying (\ref{eq:cond_period4_1}), but modifying the value of the derivatives of $p$ at those points. 

Since $S_p:E\to E$ is smooth (see Theorem A) the local character of the four cycle is governed by the differential matrix 
\begin{equation}\label{eq:lambda}
\Lambda =DS(\bar{a},\bar{b}) \ DS(\bar{b},\bar{d}) \ DS(\bar{d},\bar{c}) \ DS(\bar{c},\bar{a}),
\end{equation}
where $DS$ is the differential matrix of $S$ given in (\ref{eq:def_DS}). Substituting the parameter values and remembering that $q(x,y)$ is precisely the slope of the secant line through the point $(x,p(x)$ and $(y,p(y))$ we have
\[
\begin{array}{ll}
A (\bar{a},\bar{b}) =-\frac{1}{2}\left(1+\sqrt{5}\right)\left(1+p'(\bar{a})\right)  & \qquad B(\bar{a},\bar{b}) =\frac{1}{2}\left(3+\sqrt{5}\right) (1+p'(\bar{b}))\\
 & \\
 A (\bar{b},\bar{d}) = \frac{1}{2}\left(-2+\sqrt{5}\right)\left(3 +\sqrt{5}+2p'(\bar{b})\right) & \qquad B(\bar{d},\bar{d}) =\frac{1}{4}\left(7-3 \sqrt{5}\right) \left( 3 +\sqrt{5}+2p'(\bar{d})\right)\\
& \\
 A(\bar{d},\bar{c}) =-\frac{1}{2}\left(1+\sqrt{5}\right) \left(1+p'(\bar{d})\right)  & \qquad B(\bar{d},\bar{c}) =\frac{1}{2}\left(3+\sqrt{5}\right)  \left(1+p'(\bar{c})\right)\\
 & \\
 A (\bar{c},\bar{a}) =  \frac{1}{2}\left(-2+\sqrt{5}\right)  \left(3 +\sqrt{5}+2p'(\bar{c})\right) & \qquad B(\bar{c},\bar{a}) = \frac{1}{4}\left( 7-3\sqrt{5}\right)\left( 3 +\sqrt{5}+2p'(\bar{a})\right).
\end{array}
\]
At this point we are free to choose the values of 
$\{p'(\bar{a}),p'(\bar{b}),p'(\bar{c}),p'(\bar{d})\}$. It turs out to be the case that fixing $p'(\bar{a})=p'(\bar{b})=p'(\bar{c})=p'(\bar{d})=-1$ we get 
\[
DS(\bar{a},\bar{b})= DS (\bar{d},\bar{c}) =\left(
\begin{array}{ll}
0 & 1 \\
0 & 0
\end{array}
\right) \quad {\rm and} \quad 
DS (\bar{b},\bar{d})= DS (\bar{c},\bar{a}) =\left (
\begin{array}{cc}
0 & 1 \\
\frac{1}{2}\left(3-\sqrt{5}\right) & \sqrt{5}-2
\end{array}
\right),
\]
and hence the matrix 
\[
\Lambda =  \left(
\begin{array}{ll}
\frac{1}{4}\left(3-\sqrt{5}\right)^2 & \frac{1}{2}\left(5\sqrt{5}-11\right) \\
0 & 0
\end{array}
\right)
\]
has eigenvalues given by $0$ and $\frac{1}{4}\left(3-\sqrt{5}\right)^2 \approx 0,14589803 <1$; both of modulus less than 1. Therefore using Hermite interpolation with data 
\begin{equation*}
\begin{split}
& p(\bar{a})=\bar{d}-\bar{a}, \quad p(\bar{b})=\bar{d}-\bar{b}, \quad  p(\bar{c})=\bar{a}-\bar{c}, \quad p(\bar{d})=\bar{a}-\bar{d}, \\
& p^{\prime}(\bar{a})=p^{\prime}(\bar{b})=p^{\prime}(\bar{c})=p^{\prime}(\bar{d})=-1, \\
\end{split}
\end{equation*}
we obtain the degree seven polynomial 
\begin{equation}\label{eq:p7}
\begin{array}{ll}
p^{\star}(x) & =2.61803  - (x-1)-2.61803(x-1)^2(x-2)^2  \\
& +11.70820(x-1)^2(x-2)^2(x-2.61803)\\
& -9.23607(x-1)^2(x-2)^2(x-2.61803)^2\\ 
& +7.05573(x-1)^2(x-2)^2(x-2.61803)^2(x-3.61803).
\end{array}
\end{equation}
According to the previous arguments $S_{p^{\star}}$ exhibits an attracting four periodic orbit. In Figure \ref{fig:period4_DynamicalPlane}  we show the dynamical plane of the secant map applied to this interpolating polynomial $p^\star$.\qed

\begin{figure}
    \centering
    \subfigure[\scriptsize{Range $[-1,5]\times[-1,5]$}]{
        \includegraphics[width=0.6\textwidth]{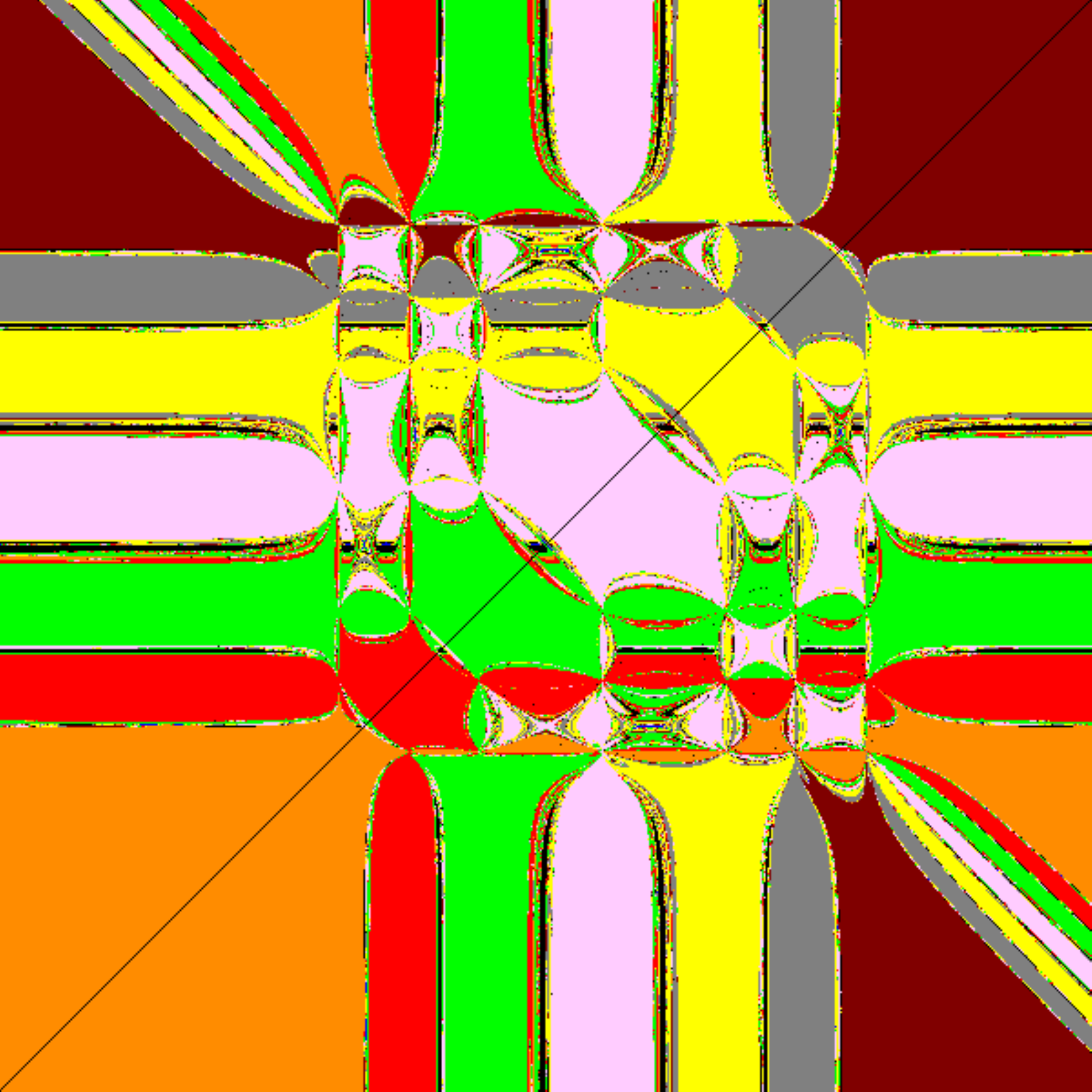}
         }
    \\
    \subfigure[]
     {
        \includegraphics[width=0.3\textwidth]{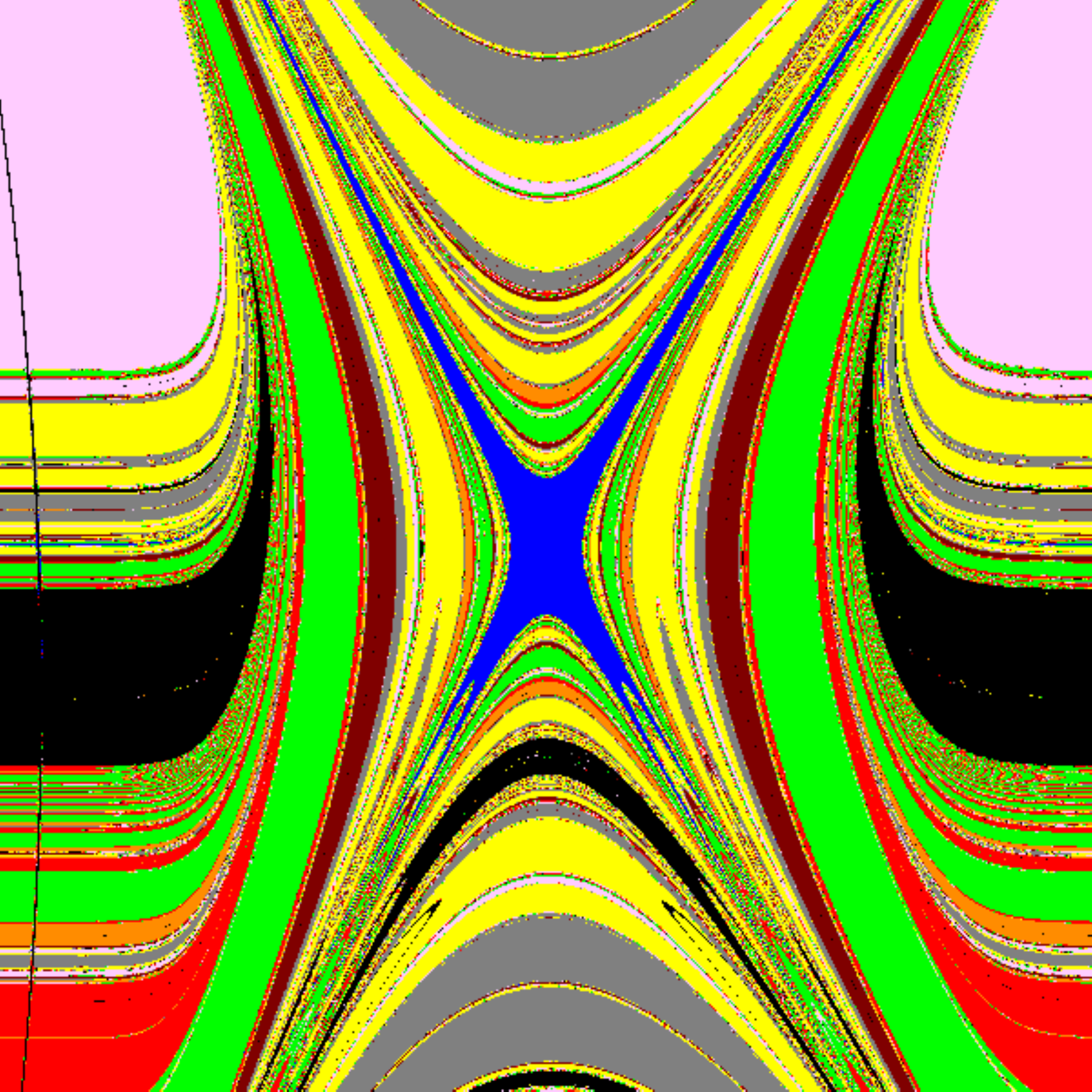}
      }
    \qquad
    \subfigure[]
    {
        \includegraphics[width=0.3\textwidth]{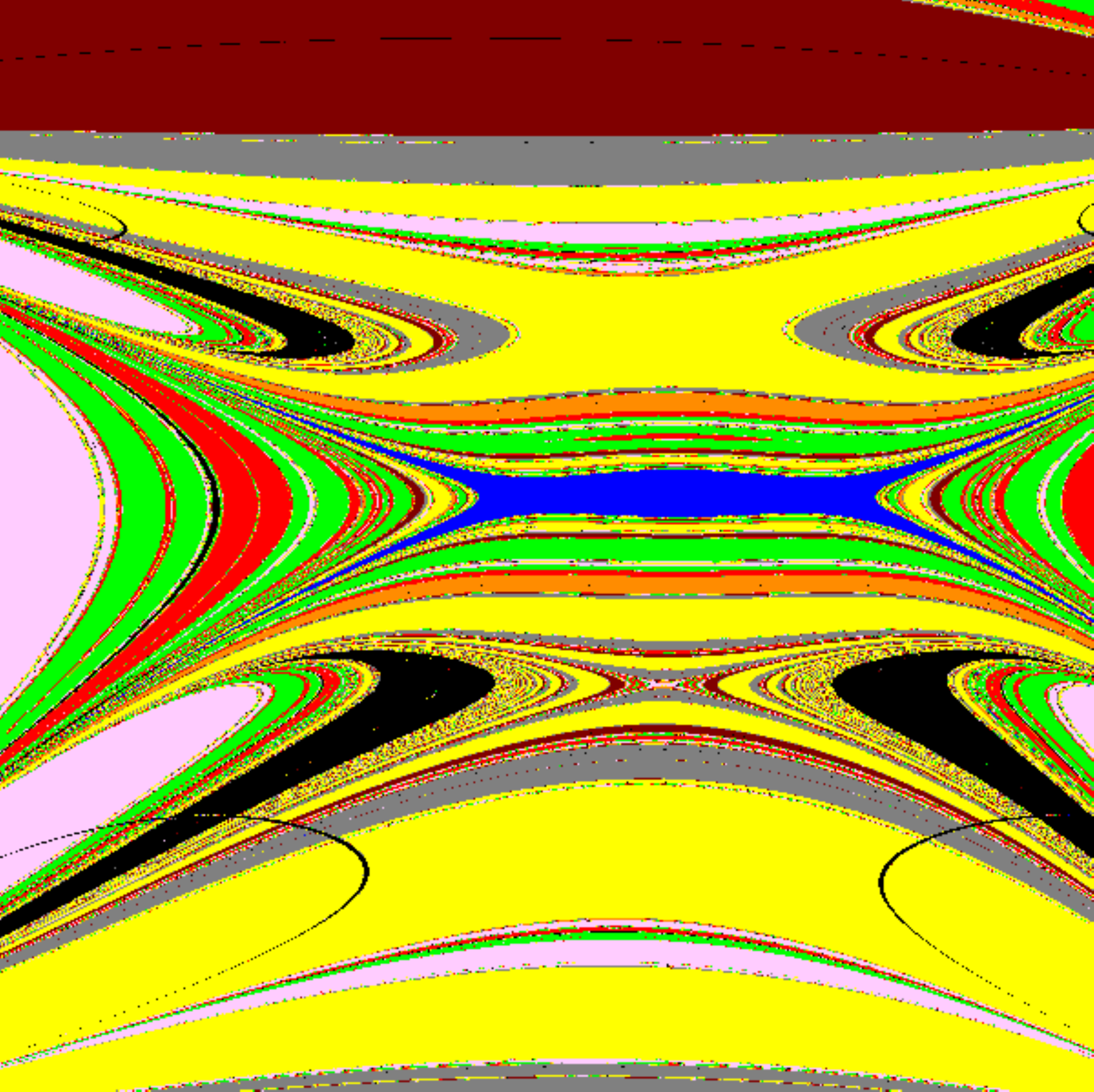}
    }
    \qquad
     \subfigure[]
    {
        \includegraphics[width=0.3\textwidth]{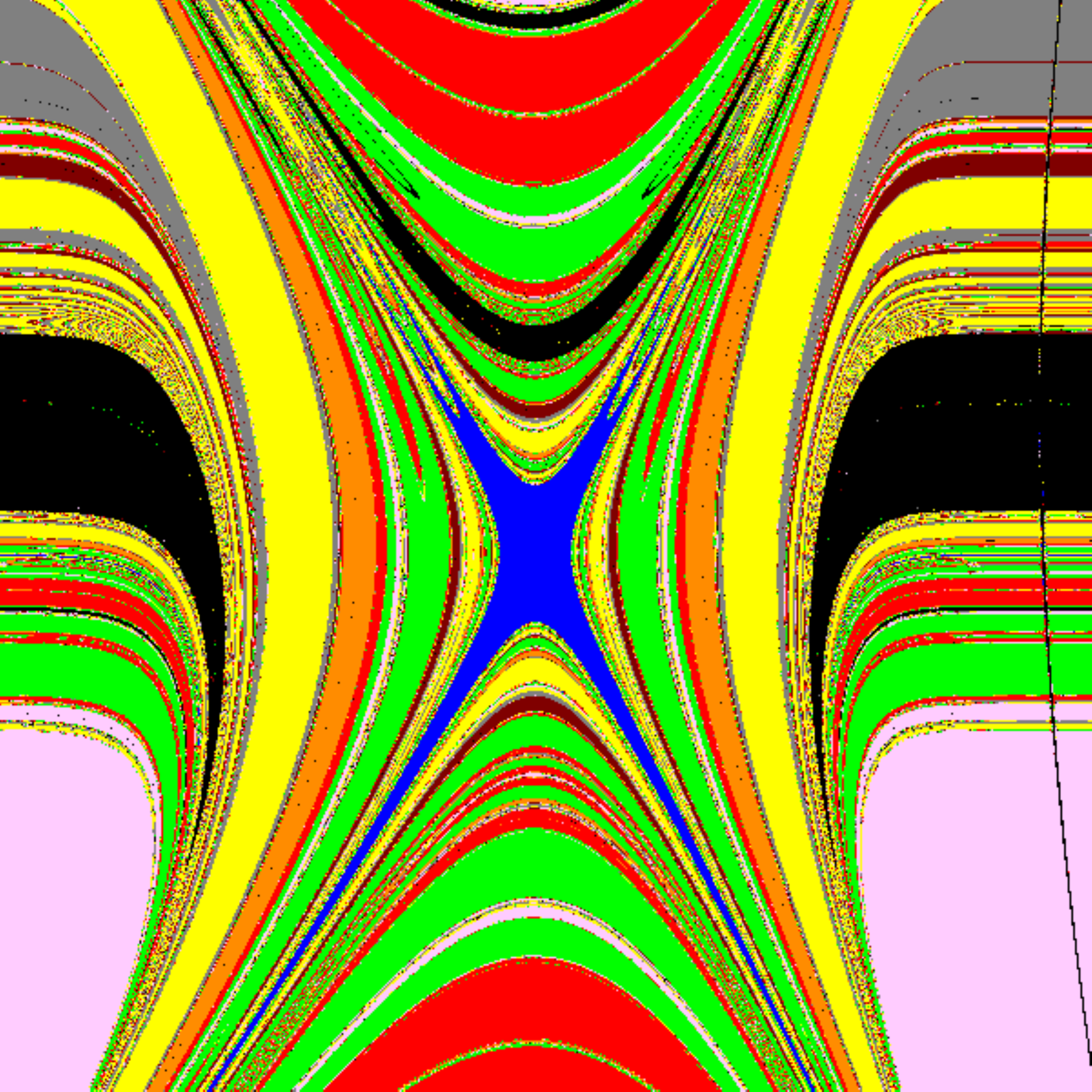}
    }
    \qquad
     \subfigure[]
    {
        \includegraphics[width=0.3\textwidth]{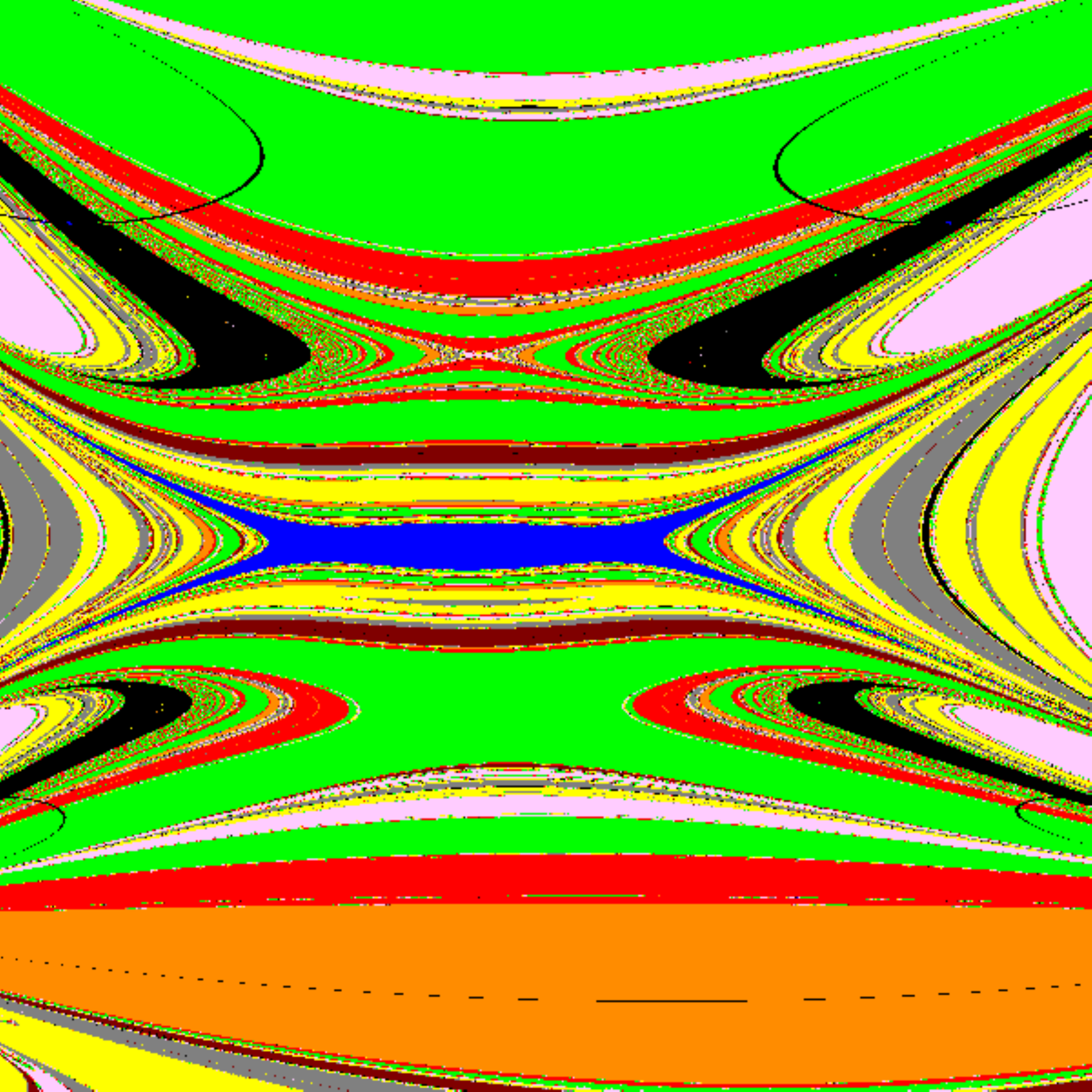}
    }
    \caption
    {
        (a) The dynamical plane of $S_{p^{\star}}$ with $p^{\star}$ given in (\ref{eq:p7}). Each basin of attraction is drawn in a different colour (yelow, orange, pink, red, brown, green, grey). In (b)-(e) images we show, in blue, the four attracting periodic orbit not corresponding to any root of $p$ as proved in Theorem B. Black regions will be explained in the next section.
Figures ((b)-(e)) correspond to zooms of (a) centered at the points $(1,2)$, $(2,(5+\sqrt{5})/2)$, $((5+\sqrt{5})/2,(3+\sqrt{5})/2)$ and $((3+\sqrt{5})/2,1)$, respectively.  
 }
    \label{fig:period4_DynamicalPlane}
\end{figure}

\section{The secant map on a torus: Proof of Theorem C}
\label{sec:torus}

The aim of this section is to extend the secant map $S$ to points in $\delta_S\setminus \mathcal{Q}$ by means of extending the map at infinity. 

\begin{lemma}\label{lem:extends}
The following statements hold.
\begin{itemize}
\item [(a)] ${\displaystyle \lim_{y \to  \pm \infty} S(x_0,y)=(\pm \infty,x_0)}$, 
\item [(b)] ${\displaystyle \lim_{x \to \pm \infty} S(x,y_0)=(y_0,y_0)}$.
\end{itemize} 
\end{lemma} 

\begin{proof}
We prove (a). From (\ref{eq:secant_real}) we have  
$$
{\displaystyle \lim_{y \to  \pm \infty}} S(x_0,y)={\displaystyle \lim_{y \to  \pm \infty}}\left(y,y- p(y) \frac{y-x_0}{p(y)-p(x_0)}\right)={\displaystyle \lim_{y \to  \pm \infty}}\left(y,\frac{p(y)x_0-yp(x_0)}{p(y)-p(x_0)}\right)=(\pm \infty,x_0),
$$
where the last equality uses that the degree of $p$ is at least 2. Statement (b) follows similarly.
\end{proof}

This lemma shows that $S$ can be extended to points of the form $(x,\pm\infty)$ and $(\pm \infty, y)$. To formalize this extension we would need to identify the symbols $+\infty$ and $-\infty$ so that the final domain of the extended map will be a torus minus one point.

If we set  $R=\{ (x,y) \in \mathbb R^2 \, ; \, 0 < x < 1 \, , \, 0 < y < 1 \}$, then a natural topological model for the torus $\mathbb T^2$ is given by $\mathbb T^2:=\overline{R}/\sim$ with the identifications $(x,0) \sim (x,1)$ and $(0,y) \sim (1,y)$. See Figure \ref{fig:torus}(a). On the other hand we might also consider $\mathbb T_0^2:=R/\sim$ which turns to be a torus minus one point $(0,0)\sim(1,0)\sim (0,1)\sim (1,1)$. See Figure \ref{fig:torus}(b).
\begin{figure}[ht]
    \centering
     \includegraphics[width=0.45\textwidth]{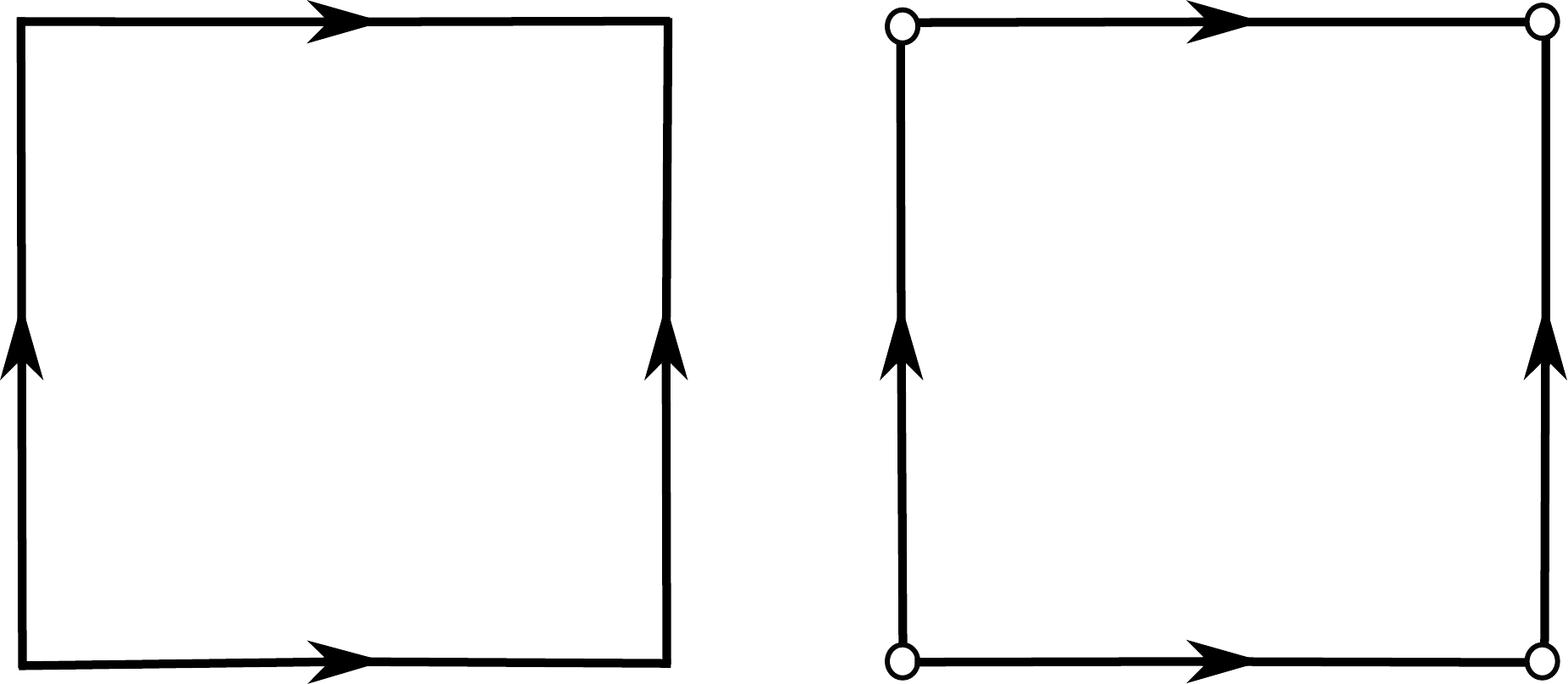}
    \put(-165,-15) {\small (a)} 
       \put(-50,-15) {\small (b)} 
           \caption{\small{The topological model of the Torus $\mathbb T^2$ and the tours minus one point (erasing the four corners of the unit square).}}
    \label{fig:torus}
    \end{figure}

This topological model together with Lemma \ref{lem:extends} adapts precisely to our goal to extend $S$ to the set   
\[
R_{\infty} = \{ (x,y) \in \mathbb R^2 \} \cup \{(x,\pm \infty) \, , \, x \in \mathbb R \} \cup \{(\pm \infty,y) \,, \, y \in \mathbb R \}
\]
\noindent with the identifications $(x,+\infty) \sim (x, - \infty)$ and $(+\infty,y) \sim (-\infty,y)$. Similarly we define $\mathbb T_{\infty}^2:=R_{\infty} / \sim$ which correspond precisely to the torus minus one point 
$$
(-\infty,-\infty)\sim(-\infty,\infty)\sim (\infty,-\infty)\sim (\infty,\infty).
$$
Equivalently,
\[
\mathbb T^2_{\infty}= \{ (x,y) \in \mathbb R^2 \} \cup \{(x, \infty) \, , \, x \in \mathbb R \} \cup \{( \infty,y) \,, \, y \in \mathbb R \}.
\]

The following three charts 
\begin{equation}\label{eq:atlas}
\begin{split}
&\varphi_1(x,y):={\rm Id}(x,y) \ {\rm if} \ (x,y)\in \mathbb R^2, \quad \varphi_2(x,y) = \left\{ \begin{array}{ll}
(x,0) & \textrm{if $y=\infty$} \\
\left(x,\frac{1}{y}\right) & \textrm{if $y\ne \{0,\infty\}$}
\end{array} \right. \ {\rm and} \\
&\varphi_3(x,y) = 
\left\{ \begin{array}{ll}
(0,y) & \textrm{if $x= \infty$}\\
\left(\frac{1}{x},y\right) & \textrm{if $x\ne \{0,\infty\}$} \ ,
\end{array} \right.
\end{split}
\end{equation}
define a smooth atlas for the surface which allow us to do the needed computations. 
Given a point $P \in \mathbb T^2_{\infty}$ we denote by ${\cal U}_P$ a small enough open neighbourhood of $P$ in $ \mathbb T^2_{\infty}$. Set 
\begin{equation}
\label{eq:def_G}
G (x,y) = \left( y, \frac{y q(x,y)-p(y)}{q(x,y)}\right), \quad (x,y)\in \mathbb R^2.
\end{equation}
Then we might use the map $G$ and the atlas $\{\varphi_j,\ j=1,2,3\}$ to extend the secant map to  $\mathbb T^2_{\infty}$. We denote the resultant map by $\hat{S}$  and its expression is given by
\begin{equation}
\label{eq:def_Shat}
\hat{S} (x,y) := \left \{ 
\begin{array}{ll}
  \left(\varphi_1^{-1} \circ G\circ \varphi_1\right)\left(x,y\right)&  \quad \hbox{ if }   \,  (x,y)\in \cal U_P \hbox{ where } P \in {\mathbb R^2\setminus \delta_S} \\
  \left(\varphi_2^{-1}\circ G\circ \varphi_1\right) \left(x,y\right)&  \quad \hbox{ if } \, (x,y)\in  \cal U_P \hbox{ where } P \in  {\delta_S \setminus \cal Q}  \\
  \left(\varphi_3^{-1}\circ G\circ \varphi_2\right) \left(x,y\right)&  \quad \hbox{ if } \, (x,y)\in  \cal U_P \hbox{ where } P =  {(x_0,\infty)} \\
   \left( \varphi_1^{-1}\circ G\circ \varphi_3\right) \left(x,y\right)&  \quad \hbox{ if } \, (x,y)\in  \cal U_P \hbox{ where } P = {(\infty,y_0)}
\end{array}
\right. 
\end{equation}

\subsection{Proof of Theorem C}
For all points in $\mathbb R^2\setminus \delta_S$ the map $\hat{S}$ (in $\mathbb R^2$-coordinates) is $G$ (see (\ref{eq:def_G})). Hence Theorem A concludes that $\hat{S}$ is smooth on this domain. To check the differentiability  of $\hat{S}$ at the points $(x,y)\ \in  \delta_S \setminus \cal Q$ we split the arguments in three cases. All computations below  follow from  (\ref{eq:atlas}), (\ref{eq:def_G}) and (\ref{eq:def_Shat}). 

\vglue 0.2truecm
\noindent {\it Case (i).} Let $(x_0,y_0) \in  \delta_S \setminus \cal Q\subset \mathbb T_{\infty}^2$ and let $\cal U_P$ be a neighbourhood of $(x_0,y_0)$. Then the map $\hat{S}$ in $\mathbb R^2$-coordinates near a point $\varphi_1\left(x_0,y_0\right)=\left(x_0,y_0\right)$ is given by
\begin{equation}\label{deriv_coord_1}
\hat{S}(x,y)=\left(y, \frac{q(x,y)}{y q(x,y)-p(y)}\right).
\end{equation}
Since $\left(x_0,y_0\right)\in \delta_S \setminus \cal Q$, it is immediate to see that on the one hand $\hat{S}(x_0,y_0)=(y_0,0)$ and on the other hand the map is locally a rational map with non zero denominator. So  $\hat{S}$ is smooth.

\vglue 0.2truecm
\noindent {\it Case (ii).} Let $(x_0,\infty) \in   \mathbb T^2_{\infty}$. Doing some computations the expression of $\hat{S}$ in $\mathbb R^2$-coordinates near a point $\varphi_2\left(x_0,\infty\right)=\left(x_0,0\right)$ is given by
\begin{equation}\label{deriv_coord_2}
\hat{S}(x,y)=\left(y,\frac{r(y)x-p(x)y^{k-1}}{r(y)-p(x)y^k}\right),
\end{equation}
where $r(y)=1+a_{k-1}y+\cdots + a_1 y^{k-1}+a_0 y^k$. We notice that on the one hand $\hat{S}(x_0,0)=(0,x_0)$  and on the other hand the map is locally a rational map with non zero denominator. So  $\hat{S}$ is smooth.

\noindent {\it Case (iii).} Let $ (\infty,y_0) \in \mathbb T^2_{\infty}$.  Doing some computations the expression of $\hat{S}$ in $\mathbb R^2$-coordinates near a point $\varphi_3\left(\infty,y_0\right)=\left(0,y_0\right)$ is given by
\begin{equation}\label{deriv_coord_3}
\hat{S}(x,y)=\left(y,\frac{p(y)x^{k-1}-yr(x)}{p(y)x^k-r(x)}\right) .
\end{equation}
We notice that on the one hand $\hat{S}(0,y_0)=(y_0,y_0)$  and on the other hand the map is locally a rational map with non zero denominator. So  $\hat{S}$ is smooth.

The existence of a periodic orbit of minimal period three on $\mathbb T_{\infty}^2$ (compare with Theorem B)  is a direct application of the  definition of $\hat{S}$. Indeed,  if $x_0$ is such that $p'\left(x_0\right)=0$ then $\left(x_0,x_0\right)\in\delta_S\setminus {\cal Q}$ and  its $\hat{S}$-orbit is given by,
$$
\hat{S}(x_0,x_0)=(x_0,\infty) , \quad  \hat{S}(x_0,\infty)=(\infty,x_0) \quad \hbox{ and }\quad \hat{S}(\infty,x_0)=(x_0,x_0).
$$

Fix $k \geq 3$. We need to  see that the eigenvalues of $D\hat{S}^3\left(x_0,x_0\right)$ are 0 and 1. To do that we use the corresponding $\mathbb R^2$-coordinates, equations (\ref{deriv_coord_1}), (\ref{deriv_coord_2}) and (\ref{deriv_coord_3}), obtaining 

$$
\begin{array}{ll}
D\hat{S}^3(x_0,x_0)=\hat{S}\left(\hat{S}^2(x_0,x_0)\right)D\hat{S}\left(\hat{S}(x_0,x_0)\right)D\hat{S}\left(x_0,x_0\right)&=\\
 \left(\begin{array}{cc}0 &1 \\0 & 1\end{array}\right)\left(\begin{array}{cc}0 &1 \\1 & 0\end{array}\right)\left(\begin{array}{cc}0 &1 \\-\frac{p''(x_0)}{p(x_0)} &-\frac{p''(x_0)}{p(x_0)}\end{array}\right)
=\left(\begin{array}{cc}0 &  1 \\ 0 & 1 \end{array}\right).
\end{array}
$$

\begin{remark}
As a consequence of Theorem C every point  $(x_0,x_0)$ with $p'(x_0)=0$ generates a periodic orbit $(x_0,x_0) \mapsto (x_0, \infty) \mapsto (\infty, x_0) \mapsto (x_0,x_0)$ which is not a priori attracting, since one of the eigenvalues is 1. However in Figure  \ref{fig:3cycle} we show in one particular example the existence of an open region, having $(x_0,x_0)$ on its boundary, of initial seeds converging to the three periodic orbit (similar regions can also be observed on the dynamical plane of previous examples; see Figures \ref{fig:dyn_plane} and \ref{fig:period4}). In  \cite{BedFri}  (Theorem 3.2) the authors give a explicit domain of seeds converging to the periodic three cycle.

\end{remark} 
 \begin{figure}[ht]
    \centering
     \includegraphics[width=0.4\textwidth]{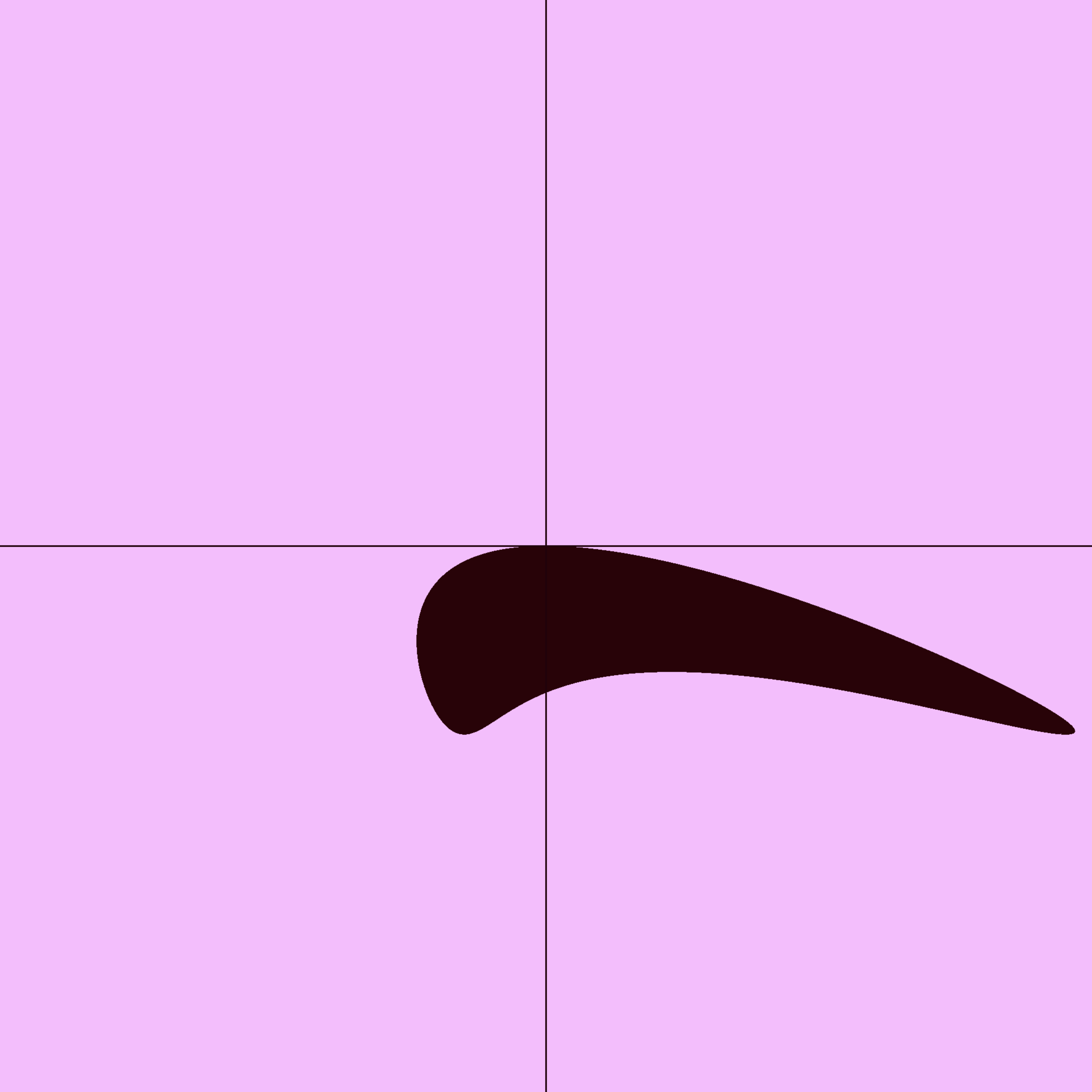}
           \caption{\small{Dynamical plane of the secant map applied to the polynomial $p(x)=\frac{1}{3}x^3- 4x+3$ near the critical point (2,2). Range of the picture $[1.92,2.08] \times [1.92,2.08]$. Points in colour black denote points converging to $(2,2)$ under $\hat{S}^{3}$, while points in colour  pink  denote points attracted by the root $x=3$ of $p$. We also show  the lines $x=2$ and $y=2$, thus the point $(2,2)$ is located at the center of the picture.}}
    \label{fig:3cycle}
    \end{figure}

    \bibliographystyle{alpha}
\bibliography{biblio}

\begin{thebibliography}{CCTV15}

\bibitem[ABFPne]{BeniniHenon}
Leandro Arosio, Anna~Miriam Benini, John~Erik Fornaess, and Han Peters.
\newblock Dynamics of transcendental {H}\'enon maps.
\newblock {\em Math. Ann.}, 2017 (Online).

\bibitem[BD05]{BedfordDillerBirational}
Eric Bedford and Jeffrey Diller.
\newblock Real and complex dynamics of a family of birational maps of the
  plane: the golden mean subshift.
\newblock {\em Amer. J. Math.}, 127(3):595--646, 2005.

\bibitem[Bed03]{BedfordBirational}
Eric Bedford.
\newblock On the dynamics of birational mappings of the plane.
\newblock {\em J. Korean Math. Soc.}, 40(3):373--390, 2003.

\bibitem[BF19]{BedFri}
Eric Bedford and Paul Frigge.
\newblock The secant method for root finding, viewed as a dynamical system.
\newblock {\em Dolomites Research Notes on Approximation, to appear}, 2019.

\bibitem[BFJK14]{ConnectivityMero}
Krzysztof Bara{\'n}ski, N{\'u}ria Fagella, Xavier Jarque, and Bogus{\l}awa
  Karpi{\'n}ska.
\newblock On the connectivity of the {J}ulia sets of meromorphic functions.
\newblock {\em Invent. Math.}, 198(3):591--636, 2014.

\bibitem[BFJK18]{BarFagJarKar}
Krzysztof Bara\'{n}ski, N\'{u}ria Fagella, Xavier Jarque, and Bogus\l~awa
  Karpi\'{n}ska.
\newblock Connectivity of {J}ulia sets of {N}ewton maps: a unified approach.
\newblock {\em Rev. Mat. Iberoam.}, 34(3):1211--1228, 2018.

\bibitem[BGM99]{PlaneDenominator1}
Gian-Italo Bischi, Laura Gardini, and Christian Mira.
\newblock Plane maps with denominator. {I}. {S}ome generic properties.
\newblock {\em Internat. J. Bifur. Chaos Appl. Sci. Engrg.}, 9(1):119--153,
  1999.

\bibitem[BGM03]{PlaneDenominator2}
Gian-Italo Bischi, Laura Gardini, and Christian Mira.
\newblock Plane maps with denominator. {II}. {N}oninvertible maps with simple
  focal points.
\newblock {\em Internat. J. Bifur. Chaos Appl. Sci. Engrg.}, 13(8):2253--2277,
  2003.

\bibitem[BGM05]{PlaneDenominator3}
Gian-Italo Bischi, Laura Gardini, and Christian Mira.
\newblock Plane maps with denominator. {III}. {N}onsimple focal points and
  related bifurcations.
\newblock {\em Internat. J. Bifur. Chaos Appl. Sci. Engrg.}, 15(2):451--496,
  2005.

\bibitem[BS91a]{BerfordSmillie1}
Eric Bedford and John Smillie.
\newblock Polynomial diffeomorphisms of {${\bf C}^2$}: currents, equilibrium
  measure and hyperbolicity.
\newblock {\em Invent. Math.}, 103(1):69--99, 1991.

\bibitem[BS91b]{BerfordSmillie2}
Eric Bedford and John Smillie.
\newblock Polynomial diffeomorphisms of {${\bf C}^2$}. {II}. {S}table manifolds
  and recurrence.
\newblock {\em J. Amer. Math. Soc.}, 4(4):657--679, 1991.

\bibitem[BS92]{BerfordSmillie3}
Eric Bedford and John Smillie.
\newblock Polynomial diffeomorphisms of {$\bold C^2$}. {III}. {E}rgodicity,
  exponents and entropy of the equilibrium measure.
\newblock {\em Math. Ann.}, 294(3):395--420, 1992.

\bibitem[BS06]{BerfordSmillie}
Eric Bedford and John Smillie.
\newblock The {H}\'enon family: the complex horseshoe locus and real parameter
  space.
\newblock In {\em Complex dynamics}, volume 396 of {\em Contemp. Math.}, pages
  21--36. Amer. Math. Soc., Providence, RI, 2006.

\bibitem[Cay79a]{cayley1}
Arthur Cayley.
\newblock Applications of the {N}ewton-{F}ourier {M}ethod to an imaginary root
  of an equation.
\newblock {\em Quat.~J.~of Pure and App.~Math.}, 16:179--185, 1879.

\bibitem[Cay79b]{cayley2}
Arthur Cayley.
\newblock The {N}ewton-{F}ourier {M}ethod imaginary problem.
\newblock {\em Am.~J.~of Math.}, 2:97, 1879.

\bibitem[Cay80]{cayley3}
Arthur Cayley.
\newblock On the {N}ewton-{F}ourier imaginary problem.
\newblock {\em Proc.~Camb.~Phil.~Soc.}, 3:231--232, 1880.

\bibitem[CCTV15]{secant_2}
Beatriz Campos, Alicia Cordero, Juan~R. Torregrosa, and Pura Vindel.
\newblock A multidimensional dynamical approach to iterative methods with
  memory.
\newblock {\em Appl. Math. Comput.}, 271:701--715, 2015.

\bibitem[CMn11]{CimaManosasBirational}
Anna Cima and Francesc Ma\~nosas.
\newblock Real dynamics of integrable birational maps.
\newblock {\em Qual. Theory Dyn. Syst.}, 10(2):247--275, 2011.

\bibitem[CZ14]{ZafarCima}
Anna Cima and Sundus Zafar.
\newblock Integrability and algebraic entropy of {$k$}-periodic non-autonomous
  {L}yness recurrences.
\newblock {\em J. Math. Anal. Appl.}, 413(1):20--34, 2014.

\bibitem[DL15]{DujardinLyubich}
Romain Dujardin and Mikhail Lyubich.
\newblock Stability and bifurcations for dissipative polynomial automorphisms
  of {$\Bbb{C}^2$}.
\newblock {\em Invent. Math.}, 200(2):439--511, 2015.

\bibitem[Duj04]{DujardinHenon}
Romain Dujardin.
\newblock H\'enon-like mappings in {$\Bbb C^2$}.
\newblock {\em Amer. J. Math.}, 126(2):439--472, 2004.

\bibitem[FsS92]{FornaesSibonyHenon}
John~Erik Forn\ae~ss and Nessim Sibony.
\newblock Complex {H}\'enon mappings in {${\bf C}^2$} and {F}atou-{B}ieberbach
  domains.
\newblock {\em Duke Math. J.}, 65(2):345--380, 1992.

\bibitem[HOV94]{HubbardHenon1}
John~H. Hubbard and Ralph~W. Oberste-Vorth.
\newblock H\'enon mappings in the complex domain. {I}. {T}he global topology of
  dynamical space.
\newblock {\em Inst. Hautes \'Etudes Sci. Publ. Math.}, (79):5--46, 1994.

\bibitem[HOV95]{HubbardHenon2}
John~H. Hubbard and Ralph~W. Oberste-Vorth.
\newblock H\'enon mappings in the complex domain. {II}. {P}rojective and
  inductive limits of polynomials.
\newblock In {\em Real and complex dynamical systems ({H}iller\o d, 1993)},
  volume 464 of {\em NATO Adv. Sci. Inst. Ser. C Math. Phys. Sci.}, pages
  89--132. Kluwer Acad. Publ., Dordrecht, 1995.

\bibitem[HPV00]{HubbardHenon3}
John Hubbard, Peter Papadopol, and Vladimir Veselov.
\newblock A compactification of {H}\'enon mappings in {${\bf C}^2$} as
  dynamical systems.
\newblock {\em Acta Math.}, 184(2):203--270, 2000.

\bibitem[HSS01]{HowToNewton}
John Hubbard, Dierk Schleicher, and Scott Sutherland.
\newblock How to find all roots of complex polynomials by {N}ewton's method.
\newblock {\em Invent. Math.}, 146(1):1--33, 2001.

\bibitem[McM87]{FamiliesRational}
Curt McMullen.
\newblock Families of rational maps and iterative root-finding algorithms.
\newblock {\em Ann. of Math. (2)}, 125(3):467--493, 1987.

\bibitem[Prz89]{przytycki}
Feliks Przytycki.
\newblock Remarks on the simple connectedness of basins of sinks for iterations
  of rational maps.
\newblock In {\em Dynamical systems and ergodic theory ({W}arsaw, 1986)},
  volume~23 of {\em Banach Center Publ.}, pages 229--235. PWN, Warsaw, 1989.

\bibitem[Shi09]{ConnectivityJulia}
Mitsuhiro Shishikura.
\newblock The connectivity of the {J}ulia set and fixed points.
\newblock In {\em Complex dynamics}, pages 257--276. A K Peters, Wellesley, MA,
  2009.

\bibitem[Tra64]{Traub_Book}
Joseph~F. Traub.
\newblock {\em Iterative Methods for the Solution of Equations}.
\newblock Prentice-Hall Series in Automatic Computation. Prentice-Hall, Inc.,
  1964.

\end{thebibliography}
\end{document}